\documentclass[aop,seceqn,citesort,MSNbibl,dvips]{arximspdf}

\doi{10.1214/10-AOP525}
\volume{38}
\issue{4}
\pubyear{2010}
\firstpage{1639}
\lastpage{1671}

\makeatletter

\newcommand{\dd}{{d}}

\newproclaim{definition}{Definition}[section]
\newtheorem{theorem}[definition]{Theorem}
\newproclaim{rmk}[definition]{Remark}
\newtheorem{prop}[definition]{Proposition}
\newtheorem{lem}[definition]{Lemma}
\newtheorem{cor}[definition]{Corollary}

\makeatother

\begin{document}
\begin{frontmatter}

\title{Almost sure invariance principle for dynamical systems by
spectral methods}
\runtitle{Invariance principle via spectral methods}

\begin{aug}
\author[A]{\fnms{S\'{e}bastien} \snm{Gou\"{e}zel}\corref{}\ead[label=e1]{sebastien.gouezel@univ-rennes1.fr}}
\runauthor{S. Gou\"{e}zel}
\affiliation{IRMAR, Universit\'{e} de Rennes 1}
\address[A]{IRMAR\\
CNRS UMR 6625\\
Universit\'{e} de Rennes 1\\
35042 Rennes\\
France\\
\printead{e1}} %adresu isvedimo komanda gale!
\end{aug}

\pdfauthor{Sebastien Gouezel}

% HISTORY:
\received{\smonth{7} \syear{2009}}
\revised{\smonth{1} \syear{2010}}

% ABSTRACT
%
\begin{abstract}
We prove the almost sure invariance principle for stationary
$\mathbb{R}^d$-valued random processes (with very precise
dimension-independent error terms), solely under a strong assumption
concerning the characteristic functions of these processes. This
assumption is easy to check for large classes of dynamical systems or
Markov chains using strong or weak spectral perturbation arguments.
\end{abstract}

% KEYWORDS
%

\end{frontmatter}

The almost sure invariance principle is a very strong
reinforcement of the central limit theorem: it ensures that the
trajectories of a process can be matched with the trajectories
of a Brownian motion in such a way that almost surely the
error between the trajectories is negligible compared to the
size of the trajectory (the result can be more or less precise,
depending on the specific error term one can obtain). These kinds of
results have a lot of consequences (see,
e.g., Melbourne and Nicol \cite{melbournenicolvectorASIP} and
references therein).

Such results are well known for one-dimensional processes,
either independent or weakly dependent (see, among many others,
Denker and Philipp \cite{denkerASIP}, Hofbauer and Keller
\cite{HofbauerKellerVar}), and for independent
higher-dimensional processes \cite{einmahlasipRd,zaitsevasipRd}.
However, for weakly dependent higher-dimensional processes,
difficulties arise since the techniques
relying on the Skorokhod representation theorem do not work
efficiently. In this direction, an approximation argument
introduced by Berkes and Philipp \cite{berkesphilipp} was recently generalized
to a large class of weakly dependent sequences in
Melbourne and Nicol \cite{melbournenicolvectorASIP}. Their results
give explicit
error terms in the vector-valued almost sure invariance
principle and are applicable when the variables under consideration
can be well approximated with respect to a suitably chosen
filtration. In particular, these results apply to a large range
of dynamical systems when they have some Markovian behavior and
sufficient hyperbolicity.

Unfortunately, it is quite common to encounter dynamical
systems for which there is no natural well-behaved filtration.
It is, nevertheless, often easy to prove classical limit
theorems, by using another class of arguments relying on
spectral theory. These arguments automatically yield a very
precise description of the characteristic functions of the
process under consideration, thereby implying limit results. It is
therefore desirable to develop an abstract argument, showing
that sufficient control on the characteristic functions of a
process implies the almost sure invariance principle for
vector-valued observables. This is our goal in this paper.
Berkes and Philipp \cite{berkesphilipp}, Theorem 5, gives such a
result, but its
assumptions are too strong for the applications we have in
mind. Moreover, even when the previous approaches are
applicable, our method gives much sharper error terms.

We will state our main probabilistic result,
Theorem \ref{main_thm_stat}, in the next section and
describe applications to dynamical systems and Markov chains in
Section \ref{sec_dynapp}. The remaining sections are devoted to
the proof of the main theorem.

%s1 ###
\section{Statement of the main result}

For $d>0$, let us consider an $\mathbb{R}^d$-valued process
$(A_0,A_1,\ldots)$, bounded in $L^p$ for some $p>2$. Under
suitable assumptions to be introduced below, we wish to show
that it can be almost surely approximated by a Brownian motion.
\begin{definition}
For $\lambda\in(0,1/2]$ and $\Sigma^2$ a (possibly degenerate)
symmetric semi-positive-definite $d\times d$ matrix, we say
that an $\mathbb{R}^d$-valued process $(A_0,A_1,\ldots)$ satisfies an
almost sure invariance principle with error exponent $\lambda$
and limiting covariance $\Sigma^2$ if there exist a probability
space $\Omega$ and two processes $(A_0^*,A_1^*,\ldots)$ and
$(B_0,B_1,\ldots)$ on $\Omega$ such that:
\begin{enumerate}
\item the processes $(A_0,A_1,\ldots)$ and $(A_0^*,A_1^*,\ldots)$
have the same distribution;
\item the random variables $B_0,B_1,\ldots$ are independent and
distributed as\break $\mathcal{N}(0,\Sigma^2)$;
\item almost surely in $\Omega$,
%
%e1.1 ###
%
\begin{equation}
\Biggl|\sum_{\ell=0}^{n-1} A_\ell^* -\sum_{\ell=0}^{n-1}B_\ell
\Biggr|=o(n^{\lambda}).
\end{equation}
\end{enumerate}
\end{definition}

A Brownian motion at integer times coincides with a sum of
i.i.d. Gaussian variables, hence this definition can also be
formulated as an almost sure approximation by a Brownian
motion, with error $o(n^{\lambda})$.

Under some assumptions on the characteristic function of
$(A_0,A_1,\ldots)$, we will prove that this process satisfies an
almost sure invariance principle. To simplify notation, for
$t\in\mathbb{R}^d$ and $x\in\mathbb{R}^d$, we will write $e^{itx}$
instead of
$e^{i\langle t,x\rangle}$.

Let us state our main assumption (\ref{Hindep}), ensuring that
the process we consider is close enough to an independent
process: there exist $\varepsilon_0>0$ and $C, c>0$ such that for
any $n,m>0$, $b_1<b_2<\cdots<b_{n+m+1}$, $k>0$ and
$t_1,\ldots,t_{n+m}\in\mathbb{R}^d$ with $|t_j|\leq\varepsilon_0$, we
have
{\renewcommand{\theequation}{H}
\begin{eqnarray}
\label{Hindep}
&& \bigl| E \bigl( e^{i\sum_{j=1}^{n} t_j (\sum_{\ell
=b_j}^{b_{j+1}-1} A_\ell) +
i\sum_{j=n+1}^{n+m}t_j (\sum_{\ell=b_{j}+k}^{b_{j+1}+k-1}A_\ell
)}
\bigr)
\nonumber\\
&&\quad{}
-E \bigl( e^{i\sum_{j=1}^{n} t_j (\sum_{\ell=b_j}^{b_{j+1}-1}
A_\ell)} \bigr)\cdot
E \bigl( e^{i\sum_{j=n+1}^{n+m}t_j (\sum_{\ell
=b_{j}+k}^{b_{j+1}+k-1}A_\ell)} \bigr) \bigr|
\\
&&\qquad
\leq C (1+{\max}|b_{j+1}-b_j|)^{C(n+m)} e^{-c k}.\nonumber\\[-12pt]\nonumber
\end{eqnarray}}
\indent This assumption says that if one groups the random variables
into $n+m$ blocks, then a gap of size $k$ between two blocks
gives characteristic functions which are exponentially close
(in terms of $k$) to independent characteristic functions, with
an error which is, for each block, polynomial in terms of the
size of the block. This control is only required for Fourier
parameters $t_j$ close to $0$.

Of course, the assumption is trivially satisfied for independent random
variables. The interesting feature of this assumption is that it
is also very easy to check for dynamical systems when the
Fourier transfer operators are well understood; see Theorem
\ref{thm_I} below.

Our main theorem follows.
\begin{theorem}
\label{main_thm_stat}
Let $(A_0,A_1,\ldots)$ be a centered $\mathbb{R}^d$-valued stationary
process, in $L^p$ for some $p>2$, satisfying (\ref{Hindep}).
Then:
\begin{enumerate}
\item the covariance matrix $\operatorname{cov}(\sum_{\ell=0}^{n-1}
A_\ell)/n$ converges to a matrix $\Sigma^2$;
\item the sequence $\sum_{\ell=0}^{n-1}A_\ell/\sqrt{n}$
converges in distribution to $\mathcal{N}(0,\Sigma^2)$;
\item the process $(A_0,A_1,\ldots)$ satisfies an almost sure invariance
principle with limiting covariance $\Sigma^2$, for any
error exponent
%
%e1.2 ###
%
\setcounter{equation}{1}
\begin{equation}
\lambda> \frac{p}{4p-4}=\frac{1}{4}+ \frac{1}{(4p-4)}.
\end{equation}
\end{enumerate}
\end{theorem}

When $p=\infty$, the condition on the error becomes $\lambda
> 1/4$, which is quite good and independent of the dimension.
This condition $\lambda> 1/4$ had previously been obtained
only for very specific classes of dynamical systems (in
particular, closed under time reversal) for real-valued
observables (see, e.g., Field, Melbourne and T{\"o}r{\"o}k
\cite{FMTasip}, Melbourne and T{\"o}r{\"o}k \cite{melbournetorok}).

If the process is not stationary, then we need an additional
assumption to ensure the (fast enough) convergence to a normal
distribution.
\begin{theorem}
\label{main_thm}
Let $(A_0,A_1,\ldots)$ be an $\mathbb{R}^d$-valued process, bounded in
$L^p$ for some $p>2$, satisfying (\ref{Hindep}). Assume,
moreover, that ${\sum}|E(A_\ell)|<\infty$ and that there exists
a matrix $\Sigma^2$ such that, for any $\alpha>0$,
%
%e1.3 ###
%
\begin{equation}
\label{covOK}
\Biggl|\operatorname{cov} \Biggl(\sum_{\ell=m}^{m+n-1} A_\ell
\Biggr) - n\Sigma
^2 \Biggr|\leq C n^\alpha,
\end{equation}
uniformly in $m,n$. The sequence
$\sum_{\ell=0}^{n-1}A_\ell/\sqrt{n}$ then converges in distribution
to $\mathcal{N}(0,\Sigma^2)$. Moreover, the process $(A_0,A_1,\ldots)$
satisfies an almost sure invariance principle, with limiting
covariance $\Sigma^2$, for any error exponent $\lambda>
p/(4p-4)$.
\end{theorem}

Theorem \ref{main_thm_stat} is, in fact, a consequence of Theorem
\ref{main_thm} since we will prove in Lemma
\ref{covsommeborne} that a stationary process satisfying
(\ref{Hindep}) always satisfies (\ref{covOK}) (moreover, this
inequality holds with $\alpha=0$).

Contrary to the results of Berkes and Philipp \cite{berkesphilipp},
our results
are dimension-independent for i.i.d. random variables (but
they are not optimal in this case---see Einmahl \cite{einmahlasipRd},
Za{\u\i}tsev \cite{zaitsevasipRd,zaitsevLp}---for i.i.d. sequences
in $L^p$,
$2<p<\infty$, the almost sure invariance principle holds for
any error exponent $\lambda\geq1/p$).

In this paper, $C$ will denote a positive constant whose
precise value is irrelevant and may change from line to line.

%s2 ###
\section{Applications}
\label{sec_dynapp}

%s2.1 ###
\subsection{Coding characteristic functions}
\label{sec_coding}

Let us first consider a very simple example: let $T(x)=2x \operatorname{mod}
1$ on the circle $S^1=\mathbb{R}/\mathbb{Z}$ and consider a Lipschitz function
$f\dvtx S^1 \to\mathbb{R}^d$ of vanishing average for Lebesgue measure. We
would like to prove an almost sure invariance principle for the
process $(f(x),f(Tx),f(T^2 x),\ldots)$, where $x$ is distributed
on $S^1$ according to Lebesgue measure. Define an operator
$\mathcal{L}_t$ on Lipschitz functions by $\mathcal{L}_t u(x)=\sum_{T(y)=x}
e^{it f(y)}u(y)/2$. It is then easy to check that for any
$t_0,\ldots,t_{n-1}$ in $\mathbb{R}^d$,
%
%e2.1 ###
%
\begin{equation}
E ( e^{i \sum_{\ell=0}^{n-1} t_\ell f \circ T^\ell} )
=\int\mathcal{L}_{t_{n-1}}\cdots\mathcal{L}_{t_0}1 (x) \,\dd x.
\end{equation}
Using the good spectral properties of the operators $\mathcal{L}_t$,
it is not very hard to show that this implies (\ref{Hindep}).

In more complicated situations, it is often possible to encode
in the same way the characteristic functions of the process
under consideration into a family of operators. However, these
operators may act on complicated Banach spaces (of
distributions or measures). It is therefore desirable to
introduce a more abstract setting that encompasses the
essential properties of such a coding, as follows.

Consider an $\mathbb{R}^d$-valued process $(A_0,A_1,\ldots)$. Let
$\mathcal{B}$ be a Banach space and let $\mathcal{L}_t$ (for $t\in
\mathbb{R}^d$,
$|t|\leq\varepsilon_0$) be linear operators acting continuously
on $\mathcal{B}$. Assume that there exist $u_0\in\mathcal{B}$ and
$\xi_0\in
\mathcal{B}'$ (the dual of $\mathcal{B}$) such that for any
$t_0,\ldots,t_{n-1}\in\mathbb{R}^d$ with $|t_j|\leq\varepsilon_0$,
%
%e2.2 ###
%
\begin{equation}
\label{eq_coding}
E ( e^{i\sum_{\ell=0}^{n-1} t_\ell A_\ell} )
=\langle\xi_0, \mathcal{L}_{t_{n-1}} \mathcal{L}_{t_{n-2}}\cdots
\mathcal{L}_{t_1} \mathcal{L}
_{t_0} u_0 \rangle.
\end{equation}
In this case, we say that the characteristic function of
$(A_0,A_1,\ldots)$ is coded by $(\mathcal{B}, (\mathcal
{L}_t)_{|t|\leq
\varepsilon_0}, u_0,\xi_0)$.

We claim that the assumption (\ref{Hindep}) follows from
suitable assumptions on the operators $\mathcal{L}_t$, which we now
describe.
\begin{enumerate}[(I2)]
\item[(I1)] One can write
$\mathcal{L}_0=\Pi+Q$, where $\Pi$ is a one-dimensional projection
and $Q$ is an operator on $\mathcal{B}$, with $Q\Pi=\Pi Q=0$ and
$\| Q^n \|_{\mathcal{B}\to\mathcal{B}} \leq C \kappa^n$ for some
$\kappa<1$.
\item[(I2)] There exists $C>0$ such that
$\| \mathcal{L}_t^n \|_{\mathcal{B}\to\mathcal{B}} \leq C$ for all
$n\in\mathbb{N}$
and all small enough $t$.
\end{enumerate}
We will denote this set of conditions by (I).
\begin{theorem}
\label{thm_I}
Let $(A_\ell)$ be a process whose characteristic function is
coded by a family of operators $(\mathcal{L}_t)$ and which is bounded
in $L^p$ for some $p>2$. Assume that \textup{(I)} holds. There then
exist $a\in\mathbb{R}^d$ and a matrix $\Sigma^2$ such that
$ (\sum_{\ell=0}^{n-1} A_\ell-na )/\sqrt{n}$ converges
to $\mathcal{N}(0,\Sigma^2)$. Moreover, the process
$(A_\ell-a)_{\ell\in\mathbb{N}}$ satisfies an almost sure invariance
principle with limiting covariance $\Sigma^2$ for any error
exponent larger than $p/(4p-4)$.
\end{theorem}

The proof will exhibit $a$ as the limit of $E(A_\ell)$, give a
formula for $\Sigma^2$ and derive the theorem from Theorem
\ref{main_thm} since (\ref{Hindep}) and (\ref{covOK}) follow
from (I). Even better, under the assumptions of
Theorem \ref{thm_I}, we have
%
%e2.3 ###
%
\begin{equation}
\label{covOK2}
\Biggl|\operatorname{cov} \Biggl(\sum_{\ell=m}^{m+n-1} A_\ell
\Biggr) - n\Sigma
^2 \Biggr|\leq C.
\end{equation}
This is proved in Lemma \ref{covsommeborne} below.
\begin{rmk}
Let us stress that the assumptions of this theorem are
significantly weaker than those of similar results in the
literature: we do not require that a perturbed eigenvalue has a
good asymptotic expansion, or even that such an eigenvalue
exists. In particular, to the best of the author's knowledge,
the central limit theorem was not known under the assumptions
of Theorem \ref{thm_I}.
\end{rmk}

Before we prove Theorem \ref{thm_I} at the end of this
section, let us describe some applications. We will explain how to
check (I) in several practical situations. Let $T\dvtx X\to X$ be
a dynamical system, let $\mu$ be a probability measure
(invariant or not) and let $f\dvtx X\to\mathbb{R}^d$. We want to study the
process $(f,f\circ T,f\circ T^2,\ldots)$.

%s2.2 ###
\subsection{Strong continuity}
\label{sec_strong}

Assume that the characteristic function of the process
$(f,f\circ T,f\circ T^2,\ldots)$ can be coded by a family of
operators $\mathcal{L}_t$ on a Banach space $\mathcal{B}$ and that the
operator $\mathcal{L}_0$ satisfies (I1), that is, it has a simple
eigenvalue at $1$, the rest of its spectrum being contained in
a disk of radius $\kappa<1$ (such an operator is said to be
\textit{quasicompact}).
\begin{prop}
\label{propI2}
If the family $\mathcal{L}_t\dvtx\mathcal{B}\to\mathcal{B}$ depends
continuously on the
parameter $t$ at $t=0$, then \textup{(I2)} is satisfied.
\end{prop}
\begin{pf}
By classical perturbation theory, the spectral picture for
$\mathcal{L}_0$ persists for small $t$: we can write
$\mathcal{L}_t=\lambda(t)\Pi_t+Q_t$, where $\lambda(t)\in\mathbb
{C}$, $\Pi_t$
is a one-dimensional projection and $\| Q_t^n \|\leq C
\kappa^n$ for some $\kappa<1$, uniformly for small $t$. If
$|\lambda(t)|\leq1$ for small $t$, then we obtain (I2).

For small $t$, we have
%
%e2.4 ###
%
\begin{eqnarray}
\label{eq_expect}
E(e^{it \sum_{\ell=0}^{n-1}f\circ T^\ell})&=&\langle\xi_0,
\mathcal{L}
_t^n u_0 \rangle
=\lambda(t)^n \langle\xi_0, \Pi_t u_0 \rangle+ \langle\xi_0,
Q_t^n u_0 \rangle
\nonumber\\[-8pt]\\[-8pt]
&=&\lambda(t)^n \langle\xi_0, \Pi_t u_0 \rangle+ O(\kappa^n).\nonumber
\end{eqnarray}
When $t\to0$, by continuity, the quantity $\langle\xi_0,
\Pi_t u_0\rangle$ converges to $\langle\xi_0, \Pi u_0
\rangle=1$; see (\ref{vaut1}) below. In particular, for small
enough $t$, $\langle\xi_0, \Pi_t u_0\rangle\not=0$. Since the
right-hand side of (\ref{eq_expect}) is bounded by $1$, this
gives $|\lambda(t)|\leq1$, completing the proof.
%\rightqed
\end{pf}

Let us be more specific. Let $T$ be an irreducible aperiodic
subshift of finite type, let $m$ be a Gibbs measure and let
$f\dvtx X\to\mathbb{R}^d$ be H\"{o}lder continuous with $\int f\,\dd m=0$. Let
$\mathcal{L}$ be the transfer operator associated with $T$, defined, by
duality, by $\int u\cdot v\circ T \,\dd m= \int\mathcal{L}u\cdot v\,\dd m$
and define perturbed operators $\mathcal{L}_t$ by
$\mathcal{L}_t(u)=\mathcal{L}(e^{itf}u)$. These operators code the
characteristic function of the process $(f,f\circ T,\ldots)$ and
depend analytically on $t$ [this follows from the series
expansion $e^{ix}=\sum(ix)^n/n!$ and the fact that the H\"{o}lder
functions form a Banach algebra]. The condition (I) is
checked in, for example, Guivarc'h and Hardy \cite{guivarchhardy},
Parry and Pollicott \cite{parrypollicott}. Hence, Theorem \ref{thm_I}
gives an
almost sure invariance principle for any error exponent greater than $1/4$.
This result is not new: it is already given in
Melbourne and Nicol \cite{melbournenicolvectorASIP}, although with a
weaker error
term.

If $T$ is an Anosov or Axiom A map and $f\dvtx X\to\mathbb{R}^d$ is H\"{o}lder
continuous, then the same result follows using symbolic
dynamics. One can also avoid it and directly apply Theorem
\ref{thm_I} to the transfer operator acting on a Banach space
$\mathcal{B}$ of distributions or distribution-like objects, as in
Baladi and Tsujii \cite{btaniso}, Gou{\"e}zel and Liverani \cite{GLAnosov2}.

Now, let $T\dvtx X\to X$ be a piecewise expanding map and assume
that the expansion dominates the complexity (in the sense of
Saussol \cite{saussol}, Lemma 2.2). This setting includes, in
particular, all piecewise expanding maps of the interval since
the complexity control is automatic in one dimension. Let
$f\dvtx X\to\mathbb{R}^d$ be $\beta$-H\"{o}lder continuous for some
$\beta\in
(0,1]$. The perturbed transfer operator $\mathcal{L}_t$ then acts
continuously on the Banach space $\mathcal{B}=V_\beta$ introduced in
Saussol \cite{saussol} and depends analytically on $t$ (since
$\mathcal{B}$
is a Banach algebra). With Theorem \ref{thm_I}, we get an
almost sure invariance principle for any error exponent greater than $1/4$.
This result was only known for $\dim(X)=1$ and $d=1$, thanks to
Hofbauer and Keller \cite{HofbauerKellerVar}.

This result also applies to coupled map lattices since
Bardet, Gou{\"e}zel and Keller \cite{BGKcoupling} shows (I) for
such maps. We should point out
that the Banach space $\mathcal{B}$ here is not a Banach space of
functions or distributions, but this is of no importance in
our abstract setting.

Assume, now, that $T$ is the time-one map of a contact Anosov
flow. Tsujii \cite{tsujiicontact} constructs a Banach space of
distributions on which the transfer operator $\mathcal{L}$ acts with a
spectral gap. If $f$ is smooth enough, then $\mathcal{L}_t:=
\mathcal{L}(e^{itf}\cdot)$ depends analytically on $t$. We therefore
obtain an almost sure invariance principle for any error
exponent greater than $1/4$. This result was known for real-valued
observables \cite{melbournetoroktimeone}, but is new for
$\mathbb{R}^d$-valued observables. However, our method does not apply
to the whole class of rapid-mixing hyperbolic flows, contrary
to the martingale arguments of
Melbourne and T{\"o}r{\"o}k \cite{melbournetoroktimeone}.

Finally, assume that $T\dvtx X\to X$ is a mixing Gibbs--Markov map
with invariant measure $m$, that is, it is Markov for a partition
$\alpha$ with infinitely many symbols and has the big image
property and H\"{o}lder distortion (this is a generalization of the
notion of a subshift of finite type to infinite alphabets, see,
e.g., Melbourne and Nicol \cite{melbournenicolvectorASIP}, Section
3.1, for
precise definitions). For $f\dvtx X \to\mathbb{R}^d$ and $a\in\alpha$, let
$Df(a)$ denote the best Lipschitz constant of $f$ on $a$.
Consider $f$ of zero average such that $\sum_{a\in\alpha} m(a)
Df(a)^\rho<\infty$ for some $\rho\in(0,1]$ (this class of
observables is very large, containing, in particular, all of the
weighted Lipschitz observables of Melbourne and Nicol \cite
{melbournenicolvectorASIP},
Section 3.2).
\begin{theorem}
If $f\in L^{p}$ for some $p>2$, then the process $(f,f\circ
T,\ldots)$ satisfies an almost sure invariance principle for any
error exponent $>p/(4p-4)$.
\end{theorem}

This follows from Gou{\"e}zel (\cite{gouezelnecessary}, Section 3.1), where
a Banach space $\mathcal{B}$ satisfying the assumptions of Proposition
\ref{propI2} is constructed.

It should be mentioned that the almost sure invariance principle is
invariant under the process of \textit{inducing}, that is, going
from a small dynamical system to a larger one. Many
hyperbolic dynamical systems can be obtained by inducing from
Gibbs--Markov maps and the previous theorem implies an almost
sure invariance principle for all of them (see
Melbourne and Nicol \cite{melbournenicolvectorASIP} for several examples).
\begin{rmk}
In such dynamical contexts (when the measure is invariant and
ergodic), the matrix $\Sigma^2$ is degenerate if and only if
$f$ is an $L^2$ coboundary in some direction. Indeed, if
$\Sigma^2$ is degenerate, then it follows from (\ref{covOK2}) that
there is a nonzero direction $t$ such that $\langle t, S_n
f\rangle$ is bounded in $L^2$. By Leonov's theorem (see,
e.g., Aaronson and Weiss \cite{aaronsonweiss}), this implies that
$\langle t,f
\rangle$ is an $L^2$ coboundary, that is, there exists $u\in L^2$ such
that $\langle t,f\rangle=u-u\circ T$ almost everywhere.
Conversely, this condition implies that $\Sigma^2$ is
degenerate.
\end{rmk}

%s2.3 ###
\subsection{Weak continuity}
\label{sec_weak}

In several situations, the strong continuity assumptions of the
previous subsection are not satisfied, while a weaker form of
continuity holds. We describe such a setting in this subsection.

Again, assume that the characteristic function of a process
$(f,f\circ T,f\circ T^2,\ldots)$ is coded by a family of
operators $\mathcal{L}_t$ on a Banach space $\mathcal{B}$ and that the
operator $\mathcal{L}_0$ satisfies (I1), that is, it is quasicompact
with a simple dominating eigenvalue at $1$.

We do \textit{not} assume that the map $t\mapsto\mathcal{L}_t$ is
continuous from a neighborhood of $0$ to the set of linear
operators on $\mathcal{B}$, hence classical perturbation theory does
not apply. Let $\mathcal{C}$ be a Banach space containing $\mathcal
{B}$ on
which the operators $\mathcal{L}_t$ act continuously and assume that
there exist $M\geq1$, $\kappa<1$ and $C>0$ such that:
\begin{enumerate}
\item for all $n\in\mathbb{N}$ and $|t|\leq\varepsilon_0$, we have
$\| \mathcal{L}_t^n \|_{\mathcal{C}\to\mathcal{C}} \leq C M^n$;
\item for all $n\in\mathbb{N}$, all $|t|\leq\varepsilon_0$ and all
$u\in\mathcal{B}$,
we have $\| \mathcal{L}_t^n u \|_{\mathcal{B}} \leq C\kappa^n
\| u \|_\mathcal{B}+CM^n \| u \|_{\mathcal{C}}$;
\item the quantity $\| \mathcal{L}_t-\mathcal{L}_0 \|_{\mathcal
{B}\to\mathcal{C}}$ tends
to $0$ when $t\to0$.
\end{enumerate}
Then Keller and Liverani \cite{kellerliverani}, Liverani \cite
{liveranismoothkl} show that, for
small enough $t$, the operator $\mathcal{L}_t$ acting on $\mathcal
{B}$ has a
simple eigenvalue $\lambda(t)$ close to $1$ and $\mathcal{L}_t$ can
be written as $\lambda(t)\Pi_t+Q_t$, where $\Pi_t$ is a
one-dimensional projection and, for some $C>0$ and
$\tilde\kappa<1$,
\begin{eqnarray*}
\| \Pi_t \|_{\mathcal{B}\to\mathcal{B}}&\leq& C,\qquad \| Q_t^n \|
_{\mathcal{B}\to
\mathcal{B}}\leq C\tilde\kappa^n,\\
\| \Pi_t-\Pi\|_{\mathcal{B}\to\mathcal{C}} &\to& 0 \qquad\mbox{when
}t\to0.
\end{eqnarray*}
Therefore, (I2) follows from the arguments in the proof of
Proposition \ref{propI2} if we can prove that $\langle\xi_0,
\Pi_t u_0\rangle\to\langle\xi_0, \Pi u_0 \rangle$ when $t\to
0$. By the last estimate in the previous equation, this is true
if $\xi_0$ belongs not only to $\mathcal{B}'$, but also to $\mathcal{C}'$,
which is usually the case.

%s2.4 ###
\subsection{Markov chains}

Consider a Markov chain $X_0,X_1,\ldots$ (with an initial
measure $\mu$ and a stationary measure $m$, possibly different
from $\mu$) on a state space $\mathcal{X}$. Also, let $f\dvtx\mathcal
{X}\to\mathbb{R}$
with $E_m(f)=0$. We want to study the process
$A_\ell=f(X_\ell)$.

Denote by $P$ the Markov operator associated with the Markov
chain and define a perturbed operator $P_t(u)=P(e^{itf} u)$.
Then
\begin{eqnarray*}
E_\mu( e^{i\sum_{\ell=0}^{n-1} t_\ell A_\ell})
&=&E_\mu\bigl( e^{i\sum_{\ell=0}^{n-2} t_\ell f(X_\ell)} \cdot
E\bigl(e^{it_{n-1}f(X_{n-1})} | X_{n-2}\bigr)\bigr)
\\
&=&E_\mu\bigl( e^{i\sum_{\ell=0}^{n-2} t_\ell f(X_\ell)} P_{t_{n-1}}1
(X_{n-2})\bigr).
\end{eqnarray*}
By induction, we obtain
%
%e2.5 ###
%
\begin{equation}
E_\mu( e^{i\sum_{\ell=0}^{n-1} t_\ell A_\ell}) = \int
P_{t_0}P_{t_1} \cdots P_{t_{n-1}} 1 \,\dd\mu.
\end{equation}
This is very similar to the coding property introduced in
(\ref{eq_coding}), the (minor) difference being that the
composition is made in the reverse direction. In particular,
the proof of Theorem \ref{thm_I} still works in this context.
We obtain the following result.
\begin{prop}
Let $\mathcal{B}$ be a Banach space of functions on $\mathcal{X}$
such that
$1\in\mathcal{B}$ and integration with respect to $\mu$ is
continuous on
$\mathcal{B}$. If the operators $P_t$ satisfy the condition
\textup{(I)} on $\mathcal{B}$, then the process $f(X_\ell)$ satisfies
(\ref{Hindep}). If $f(X_\ell)$ is bounded in $L^p$ for some
$p>2$, then it follows that the process $(f(X_\ell))$ satisfies an
almost sure invariant principle for any error exponent
$\lambda>p/(4p-4)$.
\end{prop}

To check condition (I), the arguments of Sections \ref{sec_strong} or
\ref{sec_weak} can be applied (if the Banach space $\mathcal{B}$ is
carefully chosen, depending on the properties of the random walk under
consideration). In particular, we refer the reader to the
article~\cite{hervepene}, where several examples are studied, including
uniformly ergodic chains, geometrically ergodic chains and iterated
random Lipschitz models. In particular, it is shown in this article
that the weak continuity arguments of Section \ref{sec_weak} are very
powerful in some situations where the strong continuity of Section
\ref{sec_strong} does not hold.

%s2.5 ###
\subsection[Proof of Theorem 2.1 assuming Theorem 1.3]{Proof of
Theorem \protect\ref{thm_I} assuming Theorem \protect\ref{main_thm}}

\mbox{}

\textit{Step} 1: there exists $u_1\in\mathcal{B}$ such that, for
$t_0,\ldots,t_{n-1}\in B(0, \varepsilon_0)$,
%
%e2.6 ###
%
\begin{equation}
\label{decritPi0}
\Pi( \mathcal{L}_{t_{n-1}} \cdots\mathcal{L}_{t_0} u_0)
= \langle\xi_0, \mathcal{L}_{t_{n-1}} \cdots\mathcal{L}_{t_0} u_0
\rangle u_1.
\end{equation}

Since $\Pi$ is a rank-one projection, we can write
$\Pi(u)=\langle\xi_2, u\rangle u_2$ for some $u_2\in\mathcal{B}$ and
$\xi_2\in\mathcal{B}'$ with $\langle\xi_2,u_2\rangle=1$. The trivial
equality
\[
E (e^{i\sum_{\ell=0}^{n-1} t_\ell A_\ell} )
=E (e^{i \sum_{\ell=0}^{n-1} t_\ell A_\ell+ \sum_{\ell
=n}^{n+N-1}0 \cdot A_\ell} )
\]
gives, using the coding by the operators $\mathcal{L}_t$,
\[
\langle\xi_0, \mathcal{L}_{t_{n-1}} \cdots\mathcal{L}_{t_0} u_0
\rangle
=\langle\xi_0, \mathcal{L}_0^N \mathcal{L}_{t_{n-1}} \cdots
\mathcal{L}_{t_0} u_0
\rangle.
\]
Let $u=\mathcal{L}_{t_{n-1}} \cdots\mathcal{L}_{t_0} u_0$. When $N$
tends to
infinity, $\mathcal{L}_0^N$ tends to $\Pi$. Hence, letting $N$ tend to
infinity in the previous equality, we get
%
%e2.7 ###
%
\begin{equation}
\label{eqphi0}
\langle\xi_0, u\rangle= \langle\xi_0, \Pi u\rangle
= \langle\xi_0, u_2\rangle\cdot\langle\xi_2, u \rangle.
\end{equation}
Moreover,
%
%e2.8 ###
%
\begin{eqnarray}
\label{vaut1}
\langle\xi_0, u_0\rangle &=& \langle\xi_0,\Pi u_0\rangle
=\lim\langle\xi_0, \mathcal{L}_0^N u_0\rangle\nonumber\\[-8pt]\\[-8pt]
&=&\lim E( e^{i\sum_{\ell=0}^{N-1}0\cdot A_\ell})=1.\nonumber
\end{eqnarray}
Taking $u=u_0$ in (\ref{eqphi0}), this implies, in particular, that
$\langle\xi_0,u_2 \rangle\not=0$. Finally,
\[
\Pi(u)=\langle\xi_2,u\rangle u_2=\langle\xi_0, u\rangle
u_2/\langle\xi_0, u_2\rangle.
\]
We thus obtain (\ref{decritPi0}) for $u_1=u_2/\langle\xi_0,
u_2\rangle$.

\textit{Step} 2: (\ref{Hindep}) holds.

Consider $b_1<\cdots< b_{n+m+1}$, as well as
$t_1,\ldots,t_{n+m}\in B(0,\varepsilon_0)$ and $k>0$. Then
%
%e2.9 ###
%
\begin{eqnarray}
\label{qlmskjflmqskdf}
&&
E \bigl( e^{i\sum_{j=1}^{n} t_j (\sum_{\ell=b_j}^{b_{j+1}-1}
A_\ell) +
i\sum_{j=n+1}^{n+m}t_j (\sum_{\ell=b_{j}+k}^{b_{j+1}+k-1}A_\ell)}\bigr)\hspace*{-23pt}
\nonumber\\
&&\qquad= \langle\xi_0, \mathcal
{L}_{t_{n+m}}^{b_{n+m+1}-b_{n+m}}\cdots\mathcal{L}
_{t_{n+1}}^{b_{n+2}-b_{n+1}} \mathcal{L}_0^k
\mathcal{L}_{t_{n}}^{ b_{n+1}-b_{n}} \cdots\mathcal{L}_{t_1}^{b_2-b_1}
\mathcal{L}_0^{b_1} u_0 \rangle\hspace*{-23pt}
\nonumber\\[-8pt]\\[-8pt]
&&\qquad
= \langle\xi_0, \mathcal{L}_{t_{n+m}}^{b_{n+m+1}-b_{n+m}}\cdots
\mathcal{L}
_{t_{n+1}}^{b_{n+2}-b_{n+1}} (\mathcal{L}_0^k-\Pi)
\mathcal{L}_{t_{n}}^{ b_{n+1}-b_{n}} \cdots\mathcal{L}_{t_1}^{b_2-b_1}
\mathcal{L}_0^{b_1} u_0 \rangle\hspace*{-23pt}
\nonumber\\
&&\qquad\quad{}
+ \langle\xi_0, \mathcal
{L}_{t_{n+m}}^{b_{n+m+1}-b_{n+m}}\cdots\mathcal{L}
_{t_{n+1}}^{b_{n+2}-b_{n+1}} \Pi
\mathcal{L}_{t_{n}}^{ b_{n+1}-b_{n}} \cdots\mathcal{L}_{t_1}^{b_2-b_1}
\mathcal{L}_0^{b_1} u_0 \rangle.\hspace*{-23pt}\nonumber
\end{eqnarray}
All of the operators $\mathcal{L}_{t_i}$ satisfy
$\| \mathcal{L}_{t_i}^j \|_{\mathcal{B}\to\mathcal{B}}\leq C$. Since
$\| \mathcal{L}_0^k-\Pi\|_{\mathcal{B}\to\mathcal{B}} \leq
C\kappa^k$ for some
$\kappa<1$, it follows that the term on the penultimate line in (\ref
{qlmskjflmqskdf}) is bounded by $C^{n+m}
\kappa^k$. Moreover, by (\ref{decritPi0}), the term on the last
line is
\[
\langle\xi_0, \mathcal{L}_{t_{n+m}}^{b_{n+m+1}-b_{n+m}}\cdots
\mathcal{L}
_{t_{n+1}}^{b_{n+2}-b_{n+1}} u_1 \rangle
\cdot\langle\xi_0,\mathcal{L}_{t_{n}}^{ b_{n+1}-b_{n}} \cdots
\mathcal{L}
_{t_1}^{b_2-b_1}
\mathcal{L}_0^{b_1} u_0 \rangle.
\]
The second factor in this equation is simply $E (
e^{i\sum_{j=1}^{n} t_j (\sum_{\ell=b_j}^{b_{j+1}-1}
A_\ell)} )$. Moreover,
\begin{eqnarray*}
&&E \bigl( e^{i\sum_{j=n+1}^{n+m}t_j (\sum_{\ell
=b_{j}+k}^{b_{j+1}+k-1}A_\ell)} \bigr)
\\
&&\qquad
=\langle\xi_0, \mathcal{L}_{t_{n+m}}^{b_{n+m+1}-b_{n+m}}\cdots
\mathcal{L}
_{t_{n+1}}^{b_{n+2}-b_{n+1}} \mathcal{L}_0^{b_{n+1}+k} u_0 \rangle
\\
&&\qquad
=\langle\xi_0, \mathcal{L}_{t_{n+m}}^{b_{n+m+1}-b_{n+m}}\cdots
\mathcal{L}
_{t_{n+1}}^{b_{n+2}-b_{n+1}} \Pi u_0 \rangle+ O(C^m \kappa^{b_{n+1}+k})
\\
&&\qquad
=\langle\xi_0, \mathcal{L}_{t_{n+m}}^{b_{n+m+1}-b_{n+m}}\cdots
\mathcal{L}
_{t_{n+1}}^{b_{n+2}-b_{n+1}} u_1 \rangle+ O(C^m\kappa^{b_{n+1}+k}).
\end{eqnarray*}
Therefore, the last line of (\ref{qlmskjflmqskdf}) is equal to
\[
E \bigl( e^{i\sum_{j=1}^{n} t_j (\sum_{\ell=b_j}^{b_{j+1}-1}
A_\ell)} \bigr)\cdot
E \bigl( e^{i\sum_{j=n+1}^{n+m}t_j (\sum_{\ell
=b_{j}+k}^{b_{j+1}+k-1}A_\ell)} \bigr)
+O(C^m \kappa^{b_{n+1}+k}).
\]

We have proven that the difference to be estimated to check
(\ref{Hindep}) is bounded by $C^{n+m} \kappa^k + C^m
\kappa^{b_{n+1}+k}$ for some $C>1$ and $\kappa<1$. If we write
$C=2^{C'}$ and $\kappa= e^{-c}$ for some $c,C'>0$, then this
is error is at most
\[
2\cdot2^{C'(n+m)} e^{-ck}\leq
2 \cdot(1+{\max}|b_{j+1}-b_j|)^{C'(n+m)} e^{-ck}.
\]
This proves (\ref{Hindep}).

\textit{Step} 3: there exist $a\in\mathbb{R}^d$ and $C,\delta>0$ such
that $|E(A_\ell)-a|\leq C e^{-\delta\ell}$.

Working component by component, we can, without loss of
generality, work with one-dimensional random variables.

Enriching the probability space if necessary, we can construct
a centered random variable $V$, independent of all the
$A_\ell$ and belonging to $L^p$, whose characteristic function is
supported in $B(0,\varepsilon_0)$ (see Proposition \ref{prop_V}
for the existence of $V$). Also, let $T>0$. Then
\[
E(A_\ell)=E(A_\ell+V)=E\bigl( (A_\ell+V)1_{|A_\ell+V|\geq T}\bigr)
+\int_{|x|<T} x \,\dd P_{A_\ell+V}.
\]
The first\vspace*{1pt} term is bounded by $\| A_\ell+V \|_{L^2}
\| 1_{|A_\ell+V|\geq T} \|_{L^2} \leq C P( |A_\ell+V|>T)^{1/2}
\leq C /T^{1/2}$. Let $\phi_\ell(t)=E(e^{it A_\ell})
E(e^{itV})$ be the characteristic function of $A_\ell+V$. Let
$g_T$ be the Fourier transform of $x1_{|x|<T}$. Since the
Fourier transform on $\mathbb{R}$ is an isometry up to a constant
factor $c_1$, we have $\int_{|x|<T} x \,\dd P_{A_\ell+V}=c_1\int
\overline{g_T} \phi_\ell$, hence $E(A_\ell)= c_1\int
\overline{g_T} \phi_\ell+ O( T^{-1/2})$.

We have
\begin{eqnarray*}
\phi_\ell(t)&=&\langle\xi_0, \mathcal{L}_t \mathcal{L}_0^{\ell}
u_0 \rangle E(e^{itV})
\\
&=&\langle\xi_0, \mathcal{L}_t \Pi u_0 \rangle E(e^{itV})
+\langle\xi_0, \mathcal{L}_t (\mathcal{L}_0^{\ell}-\Pi) u_0
\rangle E(e^{itV})
=:\psi(t)+r_\ell(t).
\end{eqnarray*}
The function $\psi$ is independent of $\ell$, while the
function $r_\ell(t)$ depends on $\ell$, is bounded by
$C\kappa^\ell$ and is supported in $\{|t|\leq\varepsilon_0\}$. We
obtain
\begin{eqnarray*}
E(A_\ell) &=& c_1\int\overline{g_T} \psi+ c_1\int\overline{g_T}
r_\ell+ O( T^{-1/2})
\\
&=& c_1\int\overline{g_T} \psi
+O (\| g_T \|_{L^2} \| r_\ell\|_{L^2}) + O(T^{-1/2}).
\end{eqnarray*}
The $L^2$-norm of $g_T$ is equal to $C \| x 1_{|x|<T} \|_{L^2}=C
T^{3/2}$, therefore we obtain
\[
E(A_\ell)=c_1\int\overline{g_T} \psi+O(\kappa^\ell T^{3/2})+O(T^{-1/2}).
\]
Now, consider $k,\ell\in\mathbb{N}$. Taking
$T=\kappa^{-\min(k,\ell)/3}$, we obtain, for some $\delta>0$,
\[
|E(A_\ell)-E(A_k)|\leq C e^{-\delta\min(k,\ell)}.
\]
This shows that the sequence $E(A_\ell)$ is Cauchy,
so it converges to a limit $a$. Moreover, letting $k\to
\infty$, it also yields $|E(A_\ell)-a| \leq C e^{-\delta\ell}$,
as desired.

\textit{Step} 4: conclusion of the proof.

We claim that for any $m \in\mathbb{N}$, there exists a matrix $s_m$
such that, uniformly in $\ell,m$,
%
%e2.10 ###
%
\begin{equation}
\label{msjklgsjfd}
|{\operatorname{cov}}(A_\ell, A_{\ell+m})-s_m|\leq C e^{-\delta\ell}.
\end{equation}
Since the proof is almost identical to the third step, it will
be omitted.
\begin{lem}
\label{covsommeborne}
Let $(A_\ell)$ be a process bounded in $L^p$ for some $p>2$,
satisfying (\ref{Hindep}) and satisfying (\ref{msjklgsjfd}) for some
sequence of matrices $s_m$. The series
$\Sigma^2=s_0+\sum_{m=1}^\infty(s_m+s_m^*)$ then converges in norm
and, uniformly in $m,n$,
%
%e2.11 ###
%
\begin{equation}
\label{eqcovsomme}
\Biggl|\operatorname{cov} \Biggl(\sum_{\ell=m}^{m+n-1} A_\ell
\Biggr) - n\Sigma^2
\Biggr|\leq C.
\end{equation}
\end{lem}

Let us temporarily accept this lemma. The process
$(A_\ell-a)$ then satisfies all the assumptions of Theorem
\ref{main_thm}. Theorem \ref{thm_I} follows from this theorem.
\begin{pf*}{Proof of Lemma \ref{covsommeborne}}
Let us first prove that for some $\delta>0$,
%
%e2.12 ###
%
\begin{equation}
\label{wmlkxjv}
|{\operatorname{cov}}(A_\ell, A_{\ell+m})|\leq C e^{-\delta m}.
\end{equation}
To simplify notation, we will assume that $d=1$. Although the
estimate (\ref{wmlkxjv}) follows easily from the techniques we
will develop later in this paper, we will now give a
direct elementary proof. Let $V,V'$ be two independent random
variables, as in the third step of the previous proof. Then
\[
E(A_\ell A_{\ell+m})=E\bigl( (A_\ell+V) (A_{\ell+m}+V')\bigr)
= \int x y \,\dd P(x,y),
\]
where $P$ is the distribution of $(A_\ell+V, A_{\ell+m}+V')$.
For $T>0$, we have
\begin{eqnarray*}
\int|xy|1_{|x|>T} \,\dd P(x,y)
&=&
E \bigl( |A_\ell+V| |A_{\ell+m}+V'| 1_{|A_\ell+V|>T}\bigr)
\\
&\leq& \| A_\ell+V \|_{L^p} \| A_{\ell+m}+V' \|_{L^2} \bigl\| 1_{|A_\ell
+V|>T} \bigr\|_{L^q},
\end{eqnarray*}
where $q>1$ is chosen so that $1/p+1/2+1/q=1$. Moreover,
$\| 1_{|A_\ell+V|>T} \|_{L^q}=P( |A_\ell+V|>T)^{1/q} \leq C
T^{-1/q}$. We have proven that for some $\rho>0$, we have
$\int|xy|1_{|x|>T} \,\dd P(x,y) \leq C T^{-\rho}$. In the same
way, $\int|xy|1_{|y|>T} \,\dd P(x,y) \leq C T^{-\rho}$.
Therefore,
\[
E(A_\ell A_{\ell+m})=\int xy 1_{|x|,|y|\leq T} \,\dd P(x,y)+O(T^{-\rho}).
\]
The characteristic function $\phi$ of $(A_\ell+V,
A_{\ell+m}+V')$ is given by
\[
\phi(t,u)=E(e^{it A_\ell+i u A_{\ell+m}}) E(e^{it V}) E(e^{iu V'}).
\]
It is therefore supported in $\{|t|,|u|\leq\varepsilon_0\}$.
Denoting by $h_T$ the Fourier transform of the function $xy
1_{|x|,|y|\leq T}$ and using the fact that the Fourier
transform is an isometry up to a constant factor $c_2=c_1^2$,
we get
\[
E(A_\ell A_{\ell+m})=c_2 \int\overline{h_T} \phi+ O(T^{-\rho}).
\]
Letting $\psi(t,u)=E(e^{it A_\ell})E( e^{i u A_{\ell+m}})
E(e^{it V}) E(e^{iu V'})$, a similar computation shows that
\[
E(A_\ell) E(A_{\ell+m})=c_2 \int\overline{h_T} \psi+ O(T^{-\rho}).
\]
Therefore,
\begin{eqnarray*}
|E(A_\ell A_{\ell+m}) - E(A_\ell) E(A_{\ell+m}) |
&=& c_2 \biggl| \int\overline{h_T} (\phi-\psi) \biggr|+O(T^{-\rho})
\\
&\leq& C \| h_T \|_{L^2} \| \phi-\psi\|_{L^2}+O(T^{-\rho}).
\end{eqnarray*}
The function $\phi-\psi$ is supported in $\{|t|,|u|\leq
\varepsilon_0\}$ and (\ref{Hindep}) implies that it is bounded by
$C e^{-c m}$ for some $c>0$. Moreover, $\| h_T \|_{L^2}=C
\| xy 1_{|x|,|y|\leq T} \|_{L^2}\leq CT^3$. Finally, we
obtain
\[
| E(A_\ell A_{\ell+m}) - E(A_\ell) E(A_{\ell+m}) |
\leq C e^{-c m} T^3+ CT^{-\rho}.
\]
Choosing $T=e^{cm/4}$, this gives (\ref{wmlkxjv}).

When $\ell\to\infty$, $\operatorname{cov}(A_\ell, A_{\ell+m})$
tends to $s_m$,
by assumption. Therefore, letting $\ell$ tend to infinity in
(\ref{wmlkxjv}), we get $|s_m|\leq C e^{-\delta m}$. From
(\ref{msjklgsjfd}), we obtain
%
%e2.13 ###
%
\begin{equation}
\label{eqcovklqf}
|{\operatorname{cov}}(A_\ell, A_{\ell+m})-s_m|\leq C \min( e^{-\delta
\ell},
e^{-\delta m}).
\end{equation}

We have
\begin{eqnarray*}
\operatorname{cov} \Biggl(\sum_{\ell=m}^{m+n-1} A_\ell\Biggr)
&=&\sum_{i=0}^{n-1} \operatorname{cov}(A_{i+m})
\\
&&{}+ \sum_{0\leq i < j\leq n-1} \bigl(\operatorname{cov}( A_{i+m}, A_{j+m}) +
\operatorname{cov}(
A_{i+m}, A_{j+m})^*\bigr).
\end{eqnarray*}
With (\ref{eqcovklqf}), we get
\begin{eqnarray*}
&&\Biggl| \operatorname{cov} \Biggl(\sum_{\ell=m}^{m+n-1} A_\ell
\Biggr) - \sum
_{i=0}^{n-1} s_0 - \sum_{0\leq i <
j \leq n-1} (s_{j-i}+s_{j-i}^*) \Biggr| \\
&&\qquad\leq C\sum_{i=0}^{n-1} e^{-\delta(i+m)} + C \sum_{0\leq i< j
\leq n-1} \min\bigl( e^{-\delta(i+m)}, e^{-\delta(j-i)}\bigr).
\end{eqnarray*}
Up to a multiplicative constant $C$, this is bounded by
\[
\sum_{i=0}^\infty e^{-\delta i} + \sum_{i=0}^\infty\sum
_{j=i+1}^{2i} e^{-\delta i}
+\sum_{i=0}^\infty\sum_{j=2i+1}^\infty e^{-\delta(j-i)}
<\infty.
\]
We have proven that
\[
\Biggl| \operatorname{cov} \Biggl(\sum_{\ell=m}^{m+n-1} A_\ell
\Biggr) - ns_0
-\sum_{k=1}^n (n-k) (s_k+s_k^*) \Biggr| \leq C.
\]
Since $\sum k |s_k+s_k^*| <\infty$, this proves
(\ref{eqcovsomme}).
\end{pf*}

%s3 ###
\section{Probabilistic tools}

%s3.1 ###
\subsection{Coupling}

As in Berkes and Philipp \cite{berkesphilipp}, the notion of coupling
is central
to our argument. In this subsection, we introduce this notion.

If $Z_1 \dvtx\Omega_1 \to E_1$ and $Z_2 \dvtx\Omega_2 \to E_2$ are two
random variables on two (possibly different) probability
spaces, then a \textit{coupling} between $Z_1$ and $Z_2$ is a way to
associate those random variables, usually so that this
association shows that $Z_1$ and $Z_2$ are close in some
suitable sense. Formally, a coupling between $Z_1$ and $Z_2$ is
a probability space $\Omega'$, together with two random
variables $Z'_1 \dvtx\Omega\to E_1$ and $Z'_2 \dvtx\Omega\to E_2$
such that $Z'_i$ is distributed as $Z_i$. Considering the
distribution of $(Z'_1,Z'_2)$ in $E_1\times E_2$, it follows
that one may take, without loss of generality, $\Omega=E_1 \times
E_2$, where $Z'_1$ and $Z'_2$ are the first and second projections.

The following lemma, also known as the Berkes--Philipp lemma, is
Lemma A.1 of Berkes and Philipp \cite{berkesphilipp}. It makes precise and
justifies the intuition that, given a coupling between two
random variables $Z_1$ and $Z_2$, and a coupling between $Z_2$
and another random variable $Z_3$, it is possible to
ensure that those couplings live on the same probability space,
giving a coupling between $Z_1$, $Z_2$ and $Z_3$.
\begin{lem}
\label{lem_jointwocouplings}
Let $E_i$, $i=1,2,3$, be separable Banach spaces. Let $F$ be a
distribution on $E_1\times E_2$ and let $G$ be a distribution
on $E_2 \times E_3$ such that the second marginal of $F$ equals
the first marginal of $G$. There then exist a probability space
and three random variables $Z_1,Z_2,Z_3$ defined on this space
such that the joint distribution of $Z_1$ and $Z_2$ is $F$ and
the joint distribution of $Z_2$ and $Z_3$ is $G$.
\end{lem}

As an application of this lemma, assume that two processes
$(X_1,\ldots,X_n)$ and $(Y_1,\ldots,Y_n)$ are given and that a
good coupling exists between variables $X$ and $Y$ distributed,
respectively, like $\sum X_i$ and $\sum Y_i$. There then exists
a coupling between $(X_1,\ldots,X_n)$ and $(Y_1,\ldots, Y_n)$
which realizes this coupling between $\sum X_i$ and $\sum Y_i$. It
is sufficient to build, simultaneously:
\begin{itemize}
\item the trivial coupling between $(X_1,\ldots,X_n)$ and $X$ such that
$X=\sum X_i$ almost
surely;
\item the given coupling between $X$ and $Y$;
\item the trivial coupling between $Y$ and $(Y_1,\ldots,Y_n)$
such that $Y=\sum Y_i$ almost surely.
\end{itemize}
This kind of argument will be used several times below,
without further details.

We will need the following lemma. It ensures that, to obtain a
coupling with good properties between two infinite processes
$(Z_1,Z_2,\ldots)$ and $(Z'_1,Z'_2,\ldots)$, it is sufficient to
do so for finite subsequences of these processes.
\begin{lem}
\label{weaklimit}
Let $u_n, v_n$ be two real sequences. Let $Z_n\dvtx\Omega\to E_n$
and\break $Z'_n\dvtx\Omega'\to E_n$ ($n\geq1$) be two sequences of
random variables, taking values in separable Banach spaces.
Assume that for any $N$ there exists a coupling between
$(Z_1,\ldots,Z_N)$ and $(Z'_1,\ldots,Z'_N)$ with
%
%e3.1 ###
%
\begin{equation}
\label{mlkjiwxcvumwxcv}
P(|Z_n-Z'_n|\geq u_n) \leq v_n
\end{equation}
for any $1\leq n\leq N$. There then exists a coupling between
$(Z_1,Z_2,\ldots)$ and $(Z'_1,Z'_2,\ldots)$ such that
(\ref{mlkjiwxcvumwxcv}) holds for any $n\in\mathbb{N}$.
\end{lem}
\begin{pf}
For all $N\in\mathbb{N}$, there exists a probability measure $P_N$ on
$(E_1\times\cdots\times E_N)^2$, the respective marginals of which are the
distributions of\break $(Z_1,\ldots,Z_N)$ and $(Z'_1,\ldots,Z'_N)$,
such that $P_N(|z_n-z'_n|\geq u_n) \leq v_n$ for $1\leq n\leq
N$, where $z_n$ and $z'_n$ denote the coordinates in the first
and second $E_n$ factors. Let us arbitrarily extend this
measure to a probability measure $\tilde P_N$ on $E^2$, where
$E=E_1\times E_2\times\cdots$. The sequence $\tilde P_N$ is
tight and any of its weak limits satisfies the required
property.
\end{pf}

%s3.2 ###
\subsection{Prokhorov distance}

\begin{definition}
If $P,Q$ are two probability distributions on a metric space,
define their Prokhorov distance $\pi(P,Q)$ as the smallest
$\varepsilon>0$ such that $P(B)\leq\varepsilon+Q(B^\varepsilon)$ for
any Borelian set $B$, where $B^\varepsilon$ denotes the open
$\varepsilon$-neighborhood of~$B$.
\end{definition}

This distance makes it possible to construct good couplings,
thanks to the following result, known as the Strassen--Dudley theorem
\cite{billinsgleyconvergence}, Theorem 6.9.
\begin{theorem}
\label{thm_exists_coupling}
Let $X,Y$ be two random variables taking values in a metric
space with respective distributions $P_X$ and $P_Y$. If
$\pi(P_X,P_Y)<c$, then there exists a coupling between $X$
and $Y$ such that $P(d(X,Y)>c)<c$.
\end{theorem}

We now turn to the estimation of the Prokhorov distance for
processes taking values in $\mathbb{R}^d$. Let $d>0$ and $N>0$. We
consider $\mathbb{R}^{dN}$ with the norm
\[
|(x_1,\ldots,x_N)|_N={\sup_{1\leq i\leq N}} |x_i|,
\]
where $|x|$ denotes the Euclidean norm of a point $x\in\mathbb{R}^d$.
\begin{lem}
\label{lemsmoothbasic}
There exists a constant $C(d)$ with the following property. Let
$F$ and $G$ be two probability distributions on $\mathbb{R}^{dN}$ with
characteristic functions $\phi$ and~$\gamma$. For any $T'>0$,
%
%e3.2 ###
%
\begin{equation}
\label{smoothing0}
\pi(F,G) \leq\sum_{j=1}^N F(|x_j|\geq T')+ (C(d)
T'^{d/2} )^N \biggl[\int_{\mathbb{R}^{dN}} |\phi-\gamma|^2
\biggr]^{1/2}.
\end{equation}
\end{lem}
\begin{pf}
After an approximation argument, we can assume, without loss of
generality, that $F$ and $G$ have respective densities $f$ and $g$. Then,
for any measurable set $A$,
\begin{eqnarray*}
F(A)-G(A)
&\leq&
F(A\cap{\max}|x_j|\leq T')+F( {\max}|x_j|>T')
\\
&&{}
-G(A \cap{\max}|x_j|\leq T')
\\
&\leq&
\int_{ |x_1|,\ldots,|x_N|\leq T'} |f-g| + \sum_{j=1}^N F( |x_j|>T').
\end{eqnarray*}
Therefore, $\pi(F,G)$ is bounded by last line of this
equation. To conclude, we have to estimate $\int_{
|x_1|,\ldots,|x_N|\leq T'} |f-g|$. We have
\[
\int_{ |x_1|,\ldots,|x_N|\leq T'} |f-g|
\leq\| f-g \|_{L^2} \bigl\| 1_{|x_1|,\ldots,|x_N|\leq T'} \bigr\|_{L^2}
=\| \phi-\gamma\|_{L^2} (CT')^{dN/2}
\]
since the Fourier transform is an isometry on $L^2$ up to a
factor $(2\pi)^{dN/2}$. This completes the proof.
\end{pf}

%s3.3 ###
\subsection{Classical tools}

Let us recall two classical results of probability theory that
we will need later. The first is Rosenthal's inequality
\cite{rosenthal} and the second is a weak version of the
Gal--Koksma strong law of large numbers \cite{philippstout},
Theorem A1, which will be sufficient for our purposes.
\begin{prop}
\label{prop_rosenthal}
Let $X_1,\ldots,X_n$ be independent centered real random
variables and let $p>2$. There exists a constant $C(p)$ such
that
%
%e3.3 ###
%
\begin{equation}
\Biggl\| \sum_{j=1}^n X_j \Biggr\|_{L^p} \leq C(p) \Biggl( \sum_{j=1}^n
E(X_j^2) \Biggr)^{1/2}
+ C(p) \Biggl(\sum_{j=1}^n E(|X_j|^p) \Biggr)^{1/p}.
\end{equation}
\end{prop}
\begin{prop}
\label{prop_galkoksma}
Let $X_1,X_2,\ldots$ be centered real random variables and
assume that for some $q\geq1$ and some $C>0$, for all $m, n$,
%
%e3.4 ###
%
\begin{equation}
E \Biggl|\sum_{j=m}^{m+n-1} X_j \Biggr|^2 \leq Cn ^q.
\end{equation}
For any $\alpha>0$, the sequence $\sum_{j=1}^N X_j /
N^{q/2+\alpha}$ then tends almost surely to $0$.
\end{prop}

The following proposition will be used in several forthcoming
constructions.
\begin{prop}
\label{prop_V}
There exists a symmetric random variable $V$ on $\mathbb{R}^d$,
belonging to $L^q$ for any $q>1$, whose characteristic function
is supported in the set $\{|t|\leq\varepsilon_0\}$.
\end{prop}
\begin{pf}
We start with a $C^\infty$ function $\phi$ supported in
$\{|t|\leq\varepsilon_0/2\}$ and consider its inverse Fourier
transform $f=\mathcal{F}^{-1}(\phi)$ (which is $C^\infty$ and rapidly
decreasing). Let $g=|f|^2 = \mathcal{F}^{-1}(\phi\star\tilde\phi)$,
where $\tilde\phi(t)=\phi(-t)$. Finally, let $h=g/\int g$. This is
nonnegative, has integral $1$ and its Fourier transform is
proportional to $\phi\star\tilde\phi$, hence it is supported in
$\{|t|\leq\varepsilon_0\}$. If we let $W$ and $W'$ be independent
random variables with density $h$, then $V=W-W'$ satisfies the
conclusion of the proposition.
\end{pf}

%s4 ###
\section{$L^p$ bounds}

Our goal in this section is to show the following bound.
\begin{prop}
\label{Lpbound}
Let $(A_0,A_1,\ldots)$ be a centered process, bounded in $L^p$
($p>2$) and satisfying (\ref{Hindep}). For any $\eta>0$, there
exists $C>0$ such that, for all $m,n\geq0$,
%
%e4.1 ###
%
\begin{equation}
\Biggl\| \sum_{\ell=m}^{m+n-1} A_\ell\Biggr\|_{L^{p-\eta}} \leq C n^{1/2}.
\end{equation}
\end{prop}

This kind of bound is classical for a large class of weakly
dependent sequences. The main point of this proposition is that
these bounds are established here solely under the assumption
(\ref{Hindep}) on the characteristic function of the process.

For the proof, we will approximate the process
$(A_0,A_1,\ldots)$ by an independent process, using
(\ref{Hindep}). Estimating the $L^{p-\eta}$-norm of this
process via Rosenthal's inequality (Proposition
\ref{prop_rosenthal}), this will yield the desired estimate.
\begin{lem}
\label{lemborneL2}
Let $(A_0,A_1,\ldots)$ be a centered process, bounded in $L^p$
for some $p>2$ and satisfying (\ref{Hindep}). Let
$u_n={\max_{m\in\mathbb{N}}}\| \sum_{\ell=m}^{m+n-1} A_\ell\|_{L^2}^2$.
For any $\alpha>0$, there exists $C>0$ such that $u_{a+b}\leq
u_a+u_b + C(1+a^\alpha+b^\alpha)(1+u_a^{1/2}+u_b^{1/2})$ for
any $a,b\geq1$.
\end{lem}
\begin{pf}
Let $m\in\mathbb{N}$ and $a\leq b$. Write
\[
X_1=\sum_{\ell=m}^{m+a-1}A_\ell\quad\mbox{and}
\quad
X_2=\sum_{\ell=m+a+\lfloor b^\alpha\rfloor}^{m+a+b-1}
A_\ell.
\]
Also, let $\tilde X_1=X_1+V_1$ and $\tilde
X_2=X_2+V_2$, where $V_1$ and $V_2$ are independent random
variables distributed like $V$ (constructed in Proposition
\ref{prop_V}). Finally, let $\tilde Y_1$ and $\tilde Y_2$ be
independent random variables, distributed like
$\tilde X_1$ and $\tilde X_2$, respectively.

Let us prove that for some $\delta=\delta(\alpha)>0$,
%
%e4.2 ###
%
\begin{equation}
\label{comptilde}
\pi( (\tilde X_1,\tilde X_2), (\tilde Y_1,\tilde Y_2)) < C
e^{-b^{\delta}}.
\end{equation}
Let $\phi$ and $\gamma$ denote, respectively, the
characteristic functions of $(X_1,X_2)$ and $(Y_1,Y_2)$, where
$Y_1$ and $Y_2$ are independent copies of $X_1$ and $X_2$.
Since there is a gap of size $b^\alpha$ between $X_1$ and
$X_2$, (\ref{Hindep}) ensures that for Fourier parameters less than or
equal to $\varepsilon_0$, $|\phi-\gamma|\leq C(1+b)^{C}
e^{-cb^\alpha}\leq C e^{-c' b^{\alpha}}$. We have $\tilde\phi
-\tilde\gamma= (\phi-\gamma) E(e^{i(t_1V_1+t_2V_2)})$. Since
the characteristic function of $V$ is supported in $\{|t|\leq
\varepsilon_0\}$, this shows that the characteristic functions
$\tilde\phi$ and $\tilde\gamma$ of $(\tilde X_1,\tilde X_2)$
and $(\tilde Y_1,\tilde Y_2)$, respectively, satisfy $|\tilde\phi
-\tilde
\gamma|\leq C e^{-c' b^{\alpha}}$ and are supported in
$\{|t|\leq\varepsilon_0\}$. Applying Lemma \ref{lemsmoothbasic}
with $N=2$ and $T'=e^{b^{\alpha/2}}$, we obtain
(\ref{comptilde}) [since the first terms in (\ref{smoothing0})
are bounded by $E( |\tilde X_i|)/T'\leq Cb/e^{b^{\alpha/2}}$,
while the second term is at most $C{T'}^d e^{-c'b^{\alpha}}$].

By (\ref{comptilde}) and Theorem \ref{thm_exists_coupling}, we
can construct a coupling between $(\tilde X_1,\tilde X_2)$ and
$(\tilde Y_1,\tilde Y_2)$ such that, outside a set $O$ of
measure at most $C e^{-b^\delta}$, we have $|\tilde X_i-\tilde
Y_i|\leq C e^{-b^\delta}$. Then $\| \tilde X_1+\tilde X_2 \|_{L^2}$
is bounded by
\[
\| 1_O (\tilde X_1+\tilde X_2) \|_{L^2}+
\| 1_{O^c}(\tilde X_1-\tilde Y_1+\tilde X_2-\tilde Y_2) \|_{L^2}
+\| \tilde Y_1+\tilde Y_2 \|_{L^2}.
\]
The first term is bounded by $\| 1_O \|_{L^q}\| \tilde X_1+\tilde X_2
\|_{L^p}$, where $q$ is chosen so that
$1/p+1/q=1/2$. Hence, it is at most $C e^{-b^\delta/q} b \leq
C$. The second term is bounded by $C e^{-b^\delta} \leq C$.
Finally, since $\tilde Y_1$ and $\tilde Y_2$ are independent
and centered, the last term is equal to $(E(\tilde
Y_1^2)+E(\tilde Y_2^2))^{1/2}$.

Since $\| V \|_{L^2}$ is finite, we finally obtain
\[
\| X_1+X_2 \|^2_{L^2}\leq C + E(Y_1^2)+E( Y_2^2)=C+E(X_1^2)+E(X_2^2).
\]
Taking into account the missing block
$\sum_{\ell=m+a}^{m+a+\lfloor b^\alpha\rfloor-1}A_\ell$ (whose
$L^2$-norm is at most $C b^\alpha$) and using the trivial
inequality $\| U+V \|^2_{L^2}\leq
\| U \|_{L^2}^2+\| V \|_{L^2}^2+2\| U \|_{L^2}\| V \|_{L^2}$,
we finally obtain
\begin{eqnarray*}
\Biggl\| \sum_{\ell=m}^{m+a+b-1} A_\ell\Biggr\|_{L^2}^2
&\leq&\Biggl\| \sum_{\ell=m}^{m+a-1} A_\ell\Biggr\|_{L^2}^2
+\Biggl\| \sum_{\ell=m+a}^{m+a+b-1} A_\ell\Biggr\|_{L^2}^2
\\
&&{}+ C b^{2\alpha} + C b^\alpha\Biggl( \Biggl\| \sum_{\ell=m}^{m+a-1}
A_\ell\Biggr\|_{L^2}
+\Biggl\| \sum_{\ell=m+a}^{m+a+b-1} A_\ell\Biggr\|_{L^2} \Biggr).
\end{eqnarray*}
This proves the lemma.
\end{pf}
\begin{lem}
\label{sousadditif}
Let $u_n\geq0$ satisfy
%
%e4.3 ###
%
\begin{equation}
\label{eqineqzarb}
u_{a+b}\leq u_a+u_b +
C(1+a^\alpha+b^\alpha)(1+u_a^{1/2}+u_b^{1/2})
\end{equation}
for all $a,b\geq1$ and some $C>0$, $\alpha\in(0,1/2)$. Then
$u_n=O(n)$.
\end{lem}
\begin{pf}
For any $\varepsilon>0$ and $x,y\geq0$, we have $xy\leq
\varepsilon x^2+\varepsilon^{-1}y^2$. From the
assumption, we therefore obtain
\[
u_{a+b} \leq(1+\varepsilon) (u_a+u_b) + C\varepsilon^{-1} \max
(a^{2\alpha
}, b^{2\alpha}).
\]
Let $v_k=\max_{0 \leq n < 2^{k+1}} u_n$. It follows from the
previous equation that
\[
v_{k+1} \leq(2+2\varepsilon) v_k + C\varepsilon^{-1} 2^{2\alpha k}.
\]
In particular,
we have
\[
\frac{v_{k+1}}{(2+2\varepsilon)^{k+1}} \leq\frac{v_k}{(2+2\varepsilon)^k}
+ C\varepsilon^{-1} \frac{2^{2\alpha k}}{(2+2\varepsilon)^{k+1}}.
\]
It follows inductively that $v_k/(2+2\varepsilon)^k \leq v_0+
C\varepsilon^{-1}\sum_j \frac{2^{2\alpha
j}}{(2+2\varepsilon)^{j+1}}<\infty$. Hence, for any $\varepsilon>0$,
$v_k=O ((2+2\varepsilon)^k)$, that is, for any $\rho>1$,
$u_n=O(n^\rho)$. Choosing $\rho$ close enough to $1$, we get,
from (\ref{eqineqzarb}), that $u_{a+b} \leq u_a+u_b + C
a^\beta+Cb^\beta$ for some $\beta<1$. Therefore, $v_{k+1} \leq
2v_k+ C 2^{\beta k}$. As above, we deduce that $v_k/2^k$ is
bounded, that is, $u_n=O(n)$.
\end{pf}
\begin{pf*}{Proof of Proposition \ref{Lpbound}}
Lemmas \ref{lemborneL2} and \ref{sousadditif} show that a
centered process in $L^p$ satisfying (\ref{Hindep}) satisfies
the following bound in $L^2$:
%
%e4.4 ###
%
\begin{equation}
\label{borneL2}
\Biggl\| \sum_{\ell=m}^{m+n-1}A_\ell\Biggr\|_{L^2} \leq Cn^{1/2}.
\end{equation}
Let us now show that the same bound holds in $L^{p-\eta}$ for
any $\eta>0$.

Let $\alpha=1/10$. For $n\in\mathbb{N}$, let $a=\lfloor
n^{1-\alpha}\rfloor$ and $b=\lfloor n^\alpha\rfloor$. Fixing $m\in
\mathbb{N}$, we decompose the interval $[m,m+n)$ as the union of the
intervals $I_j=[m+j a, m+(j+1)a - b^2)$ and
$I'_j=[m+(j+1)a-b^2, m+(j+1)a)$ for $0\leq j< b$, and a final
interval $J=[m+ba, m+n)$.

Write $X_j=\sum_{\ell\in I_j} A_\ell$ and $\tilde
X_j=X_j+V_j$, where the $V_j$ are independent and distributed
like $V$ constructed in Proposition \ref{prop_V}. Finally, let
$\tilde Y_0,\ldots,\tilde Y_{b-1}$ be independent random
variables such that $\tilde Y_j$ is distributed like $\tilde
X_j$. We claim, for some $\delta>0$ and any $j\leq b$, that
%
%e4.5 ###
%
\begin{equation}
\label{powixv}
\pi((\tilde X_0,\ldots, \tilde X_{j-1}), (\tilde X_0,\ldots, \tilde
X_{j-2}, \tilde Y_{j-1}))
< C e^{-n^{\delta}}.
\end{equation}
Indeed, the $\tilde X_j$ are blocks, each of length at most $n$,
and there are at most $n^{\alpha}$ blocks. Since there is a gap
of length $b^2=n^{2\alpha}$ between $X_{j-2}$ and $X_{j-1}$,
(\ref{Hindep})~shows that the difference between the
characteristic functions of the members of (\ref{powixv}) is at
most $C n^{C n^\alpha} \cdot e^{-c n^{2\alpha}}\leq C
e^{-c'n^{2\alpha}}$ (the terms $V_j$ ensure that it is
sufficient to consider Fourier parameters bounded by
$\varepsilon_0$). The estimate (\ref{powixv}) then follows from
Lemma~\ref{lemsmoothbasic} by taking $T'=e^{n^{\alpha}}$ and
$N=j$.

Summing the estimate in (\ref{powixv}) over $j$, we obtain
%
%e4.6 ###
%
\begin{equation}
\pi( (\tilde X_0,\ldots, \tilde X_{b-1}), (\tilde Y_0,\ldots, \tilde
Y_{b-1}))
< C e^{- n^\delta/2}.
\end{equation}
By the Strassen--Dudley Theorem \ref{thm_exists_coupling}, we can
therefore construct a coupling between those processes such
that, outside a set $O$ of measure at most $C
e^{-n^\delta/2}$, we have $|\tilde X_i-\tilde Y_i|\leq C
e^{-n^\delta/2}$ for $0\leq i \leq b-1$. As in the proof of
Lemma \ref{lemborneL2}, this gives
\[
\Biggl\| \sum_{j=0}^{b-1} \tilde X_j \Biggr\|_{L^{p-\eta}} \leq C + \Biggl\| \sum
_{j=0}^{b-1} \tilde Y_j \Biggr\|_{L^{p-\eta}}.
\]
Since the $\tilde Y_j$ are independent and centered,
Rosenthal's inequality (Proposition~\ref{prop_rosenthal})
applies. Let us write
$v_k={\max_{t\in\mathbb{N}}}\| \sum_{\ell=t}^{t+k-1} A_\ell\|
_{L^{p-\eta}}$. The $\tilde Y_j$ are bounded in $L^2$ by
$a^{1/2}$ [by (\ref{borneL2})] and in $L^{p-\eta}$ by
$C+v_{a-b^2}\leq Cv_{a-b^2}$. Hence,
\begin{eqnarray*}
\Biggl\| \sum_{j=0}^{b-1} \tilde Y_j \Biggr\|_{L^{p-\eta}}
&\leq& C \Biggl(\sum_{j=0}^{b-1} a \Biggr)^{1/2} + C \Biggl(\sum
_{j=0}^{b-1} v_{a-b^2}^{p-\eta} \Biggr)^{1/(p-\eta)}
\\
&\leq& C n^{1/2} + C v_{a-b^2} b^{1/(p-\eta)}.
\end{eqnarray*}
Since $\tilde X_j=X_j+V_j$ and $V_j$ is bounded by $C$ in
$L^{p-\eta}$, we get, from the two previous equations,
that
\[
\Biggl\| \sum_{j=0}^{b-1}X_j \Biggr\|_{L^{p-\eta}} \leq C n^{1/2} + C v_{a-b^2}
b^{1/(p-\eta)}.
\]
Finally,
\begin{eqnarray*}
\Biggl\| \sum_{\ell=m}^{m+n-1} A_\ell\Biggr\|_{L^{p-\eta}} &\leq& \Biggl\| \sum
_{j=0}^{b-1}X_j \Biggr\|_{L^{p-\eta}}+
\sum_{j=0}^{b-1} \sum_{\ell\in I'_j} \| A_\ell\|_{L^{p-\eta}}
+ \Biggl\| \sum_{\ell=m+ab}^{m+n-1} A_\ell\Biggr\|_{L^{p-\eta}}
\\
&\leq& C n^{1/2} + C v_{a-b^2} b^{1/(p-\eta)} + C n^{3\alpha} + v_{n-ab}.
\end{eqnarray*}
Therefore, since $3\alpha<1/2$,
we have
\[
v_n \leq Cn^{1/2} + C v_{a-b^2} b^{1/(p-\eta)} + v_{n-ab}.
\]
Moreover, $a\leq n^{1-\alpha}$, $b\leq n^\alpha$ and $n-ab \leq
a+b+1 \leq C n^{1-\alpha}$. If $v_n=O(n^r)$, then this gives
$v_n=O(n^s)$ for $s=s(r)=\max(1/2,
(1-\alpha)r+\alpha/(p-\eta))$. Starting from the trivial
estimate $v_n=O(n)$, we get $v_n=O(n^{s(1)})$, then
$v_n=O(n^{s(s(1))})$ and so on. Since $p-\eta>2$, this gives, in
finitely many steps, that $v_n=O(n^{1/2})$.
%\rightqed
\end{pf*}

%s5 ###
\section{Proof of the main theorem for nondegenerate covariance matrices}
\label{secnondeg}

In this section, we consider a process $(A_0,A_1,\ldots)$
satisfying the assumptions of Theorem \ref{main_thm} and such
that the matrix $\Sigma^2$ is nondegenerate. We will prove that
this process satisfies the conclusions of Theorem
\ref{main_thm}. Replacing, without loss of generality, $A_\ell$
by $A_\ell-E(A_\ell)$, we can assume that $A_\ell$ is centered.
If $K$ is a finite subset of $\mathbb{N}$, then we denote its
cardinality by
$|K|$.

The strategy of the proof is very classical: we subdivide the
integers into blocks with gaps between them, make the sums over
the different blocks independent using the gaps and
(\ref{Hindep}), use approximation results for sums of
independent random variables to handle the (now independent) sums
over the different blocks and, finally, show that the
fluctuations inside the blocks and the terms in the gaps do not
contribute much to the asymptotics.

The interesting feature of our approach is the choice of the
blocks. First, we subdivide $\mathbb{N}$ into the intervals
$[2^n,2^{n+1})$ and we then cut each of these intervals
following a triadic Cantor-like approach: we put a relatively
large gap in the middle, then we put slightly smaller gaps in
the middle of each half and we continue on in this way. This
procedure gives better results than the classical arguments
taking blocks along a polynomial progression: this would give
an error $p/(3p-2)$ in the theorem,
%
%More precisely, if we use the progression $I_1, J_1, I_2, J_2,
%I_3,\ldots$, where $I_n$ is of size $n^\myrho$, and $J_n$ is of
%size $n$ (to make $I_1,\ldots, I_{n-1}$ independent of $I_n$),
%then the maximum on $I_n$ satisfies $P(\max> n^{\myrho/2+1/p})
%J_{n}$ is of order $(1+\ldots+n)^{1/2}=n$. Hence, the overall
%error is at least $O(N+ N^{\myrho/2+1/p})$, where $n=N^{1+\myrho}$.
%Optimized by taking $\myrho=(2p-2)/p$. This gives an error term
%$n^{1/(1+\myrho)}=n^{p/(3p-2)}$, i.e., an error in the almost
%sure invariance principle of order $p/(3p-2)=1/3+2/(9p-6)$.
%Since $p/(3p-2)>p/(4p-4)$ for any $p>2$, this is worse than
%Theorem \ref{main_thm}.
%
while we obtain the better error term $p/(4p-4)$ with the
Cantor-like decomposition. The reason is that, to create $n$
manageable blocks, the classical arguments require gaps whose
union is of order $n^2$, while the triadic decomposition only
uses gaps whose union is of order $n$.

We will now describe the triadic procedure more precisely. Fix
$\beta\in(0,1)$ and $\varepsilon\in(0, 1-\beta)$. Let
$f=f(n)=\lfloor\beta n\rfloor$. We decompose $[2^n, 2^{n+1})$
as a union of $F=2^f$ intervals $(I_{n,j})_{0\leq j < F}$ of
the same length, and $F$ gaps $(J_{n,j})_{0\leq j< F}$ between
them, used to ensure sufficient independence. Good intervals and gaps
are placed alternatively, and in increasing order, as follows:
$[2^n, 2^{n+1}) = J_{n,0} \cup I_{n,0} \cup J_{n,1} \cup
I_{n,1} \cdots\cup J_{n,F-1} \cup I_{n,F-1}$.

The lengths of the gaps $J_{n,j}$ are chosen as follows. The
middle interval $J_{n,F/2}$ has length $2^{\lfloor\varepsilon n
\rfloor} 2^{f-1}$. It cuts the interval $[2^n,2^{n+1})$ into
two parts. The middle intervals of each of these parts, that is,
$J_{n,F/4}$ and $J_{n,3 F/4}$, have length $2^{\lfloor\varepsilon
n \rfloor} 2^{f-2}$. The middle intervals of the remaining four
parts have length $2^{\lfloor\varepsilon n\rfloor} 2^{f-3}$, and
so on. More formally, for $1\leq j< F$, we write
$j=\sum_{k=0}^{f-1} \alpha_k(j) 2^k$, where $\alpha_k(j)\in
\{0,1\}$, and consider the smallest number $r$ with
$\alpha_r(j)\not=0$. The length of $J_{n,j}$ is then
$2^{\lfloor\varepsilon n \rfloor} 2^{r}$. We say that this
interval is of rank $r$. This defines the length of all the
intervals $J_{n,j}$, except for $j=0$. We let
$|J_{n,0}|=2^{\lfloor\varepsilon n \rfloor} 2^f$ and say that
this interval has rank $f$.

Since there are $2^{f-1-r}$ intervals of rank $r$ for $r<f$,
with length $2^{\lfloor\varepsilon n \rfloor} 2^{r}$, the lengths
of the intervals $(J_{n,j})_{0\leq j < F}$ add up to
%
%e5.1 ###
%
\begin{equation}
\label{sommeinterv}
|J_{n,0}|+\sum_{r=0}^{f-1} 2^{\lfloor\varepsilon n \rfloor} 2^{r}\cdot2^{f-1-r}
=2^{\lfloor\varepsilon n \rfloor} 2^{f-1} (f+2).
\end{equation}
Let $|I_{n,j}|=2^{n-f} - (f+2)2^{ \lfloor\varepsilon n \rfloor
-1}$. This is a positive integer if $n$ is large enough and
${\sum}|I_{n,j}|+{\sum}|J_{n,j}|=2^n$, that is, those intervals
exactly fill $[2^n,2^{n+1})$. We will denote by $i_{n,j}$ the
smallest element of $I_{n,j}$.

We will use the lexicographical order $\prec$ on the set
$\{(n,j) \mid n\in\mathbb{N}, 0\leq j< F(n)\}$. It can also be
described as follows: $(n,j)\prec(n',j')$ if the interval
$I_{n,j}$ is to the left of $I_{n',j'}$. A sequence $(n_k,j_k)$
tends to infinity for this order if and only if $n_k\to
\infty$.

Let\vspace*{2pt} $X_{n,j}=\sum_{\ell\in I_{n,j}} A_\ell$ for $n\in\mathbb{N}$ and
$0\leq j < F(n)$. Finally, write $\mathcal{I}=\bigcup_{n,j}I_{n,j}$ and
$\mathcal{J}=\bigcup_{n,j} J_{n,j}$. The main steps of the proof are
the following:
\begin{enumerate}
\item there exists a coupling between $(X_{n,j})$ and a
sequence of independent random variables $(Y_{n,j})$, with
$Y_{n,j}$ distributed like $X_{n,j}$, such that almost
surely when $(n,j)\to\infty$,
\[
\biggl| \sum_{(n',j')\prec(n,j)} X_{n',j'} - Y_{n',j'} \biggr|
=o\bigl( 2^{(\beta+\varepsilon) n/2}\bigr);
\]
\item there exists a coupling between $(Y_{n,j})$ and a
sequence of independent Gaussian random variables
$Z_{n,j}$, with $\operatorname{cov}(Z_{n,j})=|I_{n,j}| \Sigma^2$, such
that almost surely when $(n,j)\to\infty$,
\[
\biggl| \sum_{(n',j')\prec(n,j)} Y_{n',j'} -Z_{n',j'} \biggr|
=o\bigl( 2^{(\beta+\varepsilon) n/2}+2^{((1-\beta)/2+\beta/p+\varepsilon)n}\bigr);
\]
\item coupling the $X_{n,j}$ with the $Z_{n,j}$, by virtue of the
first two steps, and writing $Z_{n,j}$ as the sum of
$|I_{n,j}|$ Gaussian random variables $\mathcal{N}(0,\Sigma^2)$,
we obtain (using Lemma \ref{lem_jointwocouplings} and the
example that follows it) a coupling between
$(A_\ell)_{\ell\in\mathcal{I}}$ and $(B_\ell)_{\ell\in\mathcal{I}}$,
where the $B_\ell$ are i.i.d. and distributed like
$\mathcal{N}(0,\Sigma^2)$, such that, when $(n,j)$ tends to
infinity,
we have
\[
\biggl|\sum_{\ell<i_{n,j}, \ell\in\mathcal{I}} A_\ell-B_\ell
\biggr|
=o\bigl(2^{(\beta+\varepsilon)n/2}+2^{((1-\beta)/2+\beta/p+\varepsilon)n}\bigr);
\]
\item we check that almost surely when $(n,j)\to\infty$,
\[
\max_{m<|I_{n,j}|} \Biggl|\sum_{\ell=i_{n,j}}^{i_{n,j}+m} A_\ell
\Biggr| = o\bigl(2^{((1-\beta)/2+\beta/p+\varepsilon)n}\bigr)
\]
and, moreover, that a similar estimate holds for the $B_\ell$'s;
\item combining the last two steps, we get that when $k$ tends to infinity,
\[
\biggl|\sum_{\ell<k, \ell\in\mathcal{I}} A_\ell-B_\ell\biggr|
=o\bigl(k^{(\beta+\varepsilon)/2}+k^{(1-\beta)/2+\beta/p+\varepsilon}\bigr);
\]
\item finally, we prove that the gaps can be neglected, that is, almost
surely
%
%e5.2 ###
%
\begin{equation}
\label{step6}
\sum_{\ell<k, \ell\in\mathcal{J}}A_\ell=o(k^{\beta/2+\varepsilon})
\end{equation}
and a similar estimate holds for the $B_\ell$'s.
\end{enumerate}
Altogether, this gives a coupling for which almost surely
\[
\biggl|\sum_{\ell<k} A_\ell-B_\ell\biggr|
=o\bigl(k^{\beta/2+\varepsilon}+k^{(1-\beta)/2+\beta/p+\varepsilon}\bigr).
\]
Let us choose $\beta$ such that the two error terms are equal,
that is, $\beta=p/(2p-2)$. We obtain an almost sure invariance
principle with error term $p/(4p-4)+\varepsilon$ for any
$\varepsilon>0$. Since the almost sure invariance principle
implies the central limit theorem, this proves Theorem
\ref{main_thm}, under the assumption that $\Sigma^2$ is
nondegenerate.

It remains to justify Steps 1, 2, 4 and 6 since
Steps 3 and 5 are trivial. This is done in the following
subsections.

%s5.1 ###
\subsection{Step 1: Coupling with independent random variables}

In this subsection, we justify the first step of the proof of
Theorem \ref{main_thm} with the following proposition.
\begin{prop}
\label{prop_make_all_indep}
There exists a coupling between $(X_{n,j})$ and $(Y_{n,j})$
such that, almost surely, when $(n,j)$ tends to infinity,
\[
\biggl| \sum_{(n',j')\prec(n,j)} X_{n',j'} - Y_{n',j'} \biggr|
=o\bigl( 2^{(\beta+\varepsilon) n/2}\bigr).
\]
\end{prop}

The rest of this subsection is devoted to the proof of this
proposition.

Let $V_{n,j}$, for $n,j\in\mathbb{N}$, be independent copies of $V$
(constructed in Proposition~\ref{prop_V}), independent of
everything else (we may need to enlarge the space to ensure
their existence). Let $\tilde X_{n,j}=X_{n,j}+V_{n,j}$.

We will write $X_n=(X_{n,j})_{0\leq j< F(n)}$ and $\tilde
X_n=(\tilde X_{n,j})_{0\leq j< F(n)}$. The proof of
Proposition~\ref{prop_make_all_indep} has two parts: first, we make the
different $\tilde X_n$ independent of each other, using the
gaps $J_{n,0}$; then, inside each block $\tilde X_n$, we make
the variables $\tilde X_{n,j}$ independent by using the smaller
gaps $J_{n,j}$.
\begin{lem}
\label{lem1kqljsdf}
Let $\tilde Q_n$ be a random variable distributed like $\tilde
X_n$, but independent of $(\tilde X_1,\ldots,\tilde X_{n-1})$. We
have
%
%e5.3 ###
%
\begin{equation}
\label{eq_firststep}
\pi((\tilde X_1,\ldots, \tilde X_{n-1},\tilde X_n),(\tilde X_1,\ldots
, \tilde X_{n-1},\tilde Q_n))
< C 4^{-n}.
\end{equation}
\end{lem}
\begin{pf}
The random process $(X_1,\ldots, X_n)$ takes its values in
$\mathbb{R}^{dD}$ for $D=\sum_{m=1}^n F(m) \leq\sum_{m=1}^{n}
2^{\beta m} \leq C 2^{\beta n}$. Moreover, each component in
$\mathbb{R}^d$ of this process is one of the $X_{n,j}$, hence it is a
sum of at most $2^n$ consecutive variables~$A_\ell$. On the
other hand, the interval $J_{n,0}$ is a gap between
$(X_j)_{j<n}$ and $X_{n}$, and its length $k$ is $C^{\pm1}
2^{\varepsilon n + \beta n}$. Let $\phi$ and $\gamma\dvtx\mathbb
{R}^{dD}\to
\mathbb{C}$ denote the respective characteristic functions of
$(X_1,\ldots,X_{n-1}, X_n)$ and $(X_1,\ldots,X_{n-1}, Q_n)$,
where $Q_n$ is distributed like $X_n$ and is independent of
$(X_1,\ldots,X_{n-1})$. The assumption (\ref{Hindep}) ensures
that for Fourier parameters $t_{m,j}$ all bounded by
$\varepsilon_0$, we have
\[
|\phi-\gamma|\leq C (1+2^n)^{C D} e^{-c k}
\leq C 2^{nC 2^{\beta n}} e^{-c 2^{\beta n+\varepsilon n}}
\leq C e^{-c' 2^{\beta n+\varepsilon n}},
\]
if $n$ is large enough.

Let $\tilde\phi$ and $\tilde\gamma$ be the characteristic
functions of, respectively, $(\tilde X_1,\ldots,\tilde X_n)$ and
$(\tilde X_1,\ldots,\tilde Q_n)$: they are obtained by
multiplying $\phi$ and $\gamma$ by the characteristic function
of $V$ in each variable. Since this function is supported in
$\{|t|\leq\varepsilon_0\}$, we obtain, in particular, that
%
%e5.4 ###
%
\begin{equation}
|\tilde\phi-\tilde\gamma|\leq C e^{-c 2^{\beta n+\varepsilon n}}.
\end{equation}

We then use Lemma \ref{lemsmoothbasic} with $N=D$ and
$T'=e^{2^{\varepsilon n/2}}$ to get
\begin{eqnarray*}
&&\pi((\tilde X_1,\ldots,\tilde X_n), (\tilde X_1,\ldots,\tilde
X_{n-1},\tilde Q_n))
\\
&&\qquad\leq\sum_{m\leq n} \sum_{j< F(m)} P( |\tilde X_{m,j}|\geq
e^{2^{\varepsilon n/2}})
+ e^{C D 2^{\varepsilon n/2}} e^{-c 2^{\beta n+\varepsilon n}}.
\end{eqnarray*}
The second term is, again, bounded by $e^{-c'2^{\beta n+\varepsilon
n}}$, while each term in the first sum is bounded by
$e^{-2^{\varepsilon n/2}}E(|\tilde X_{m,j}|) \leq e^{-2^{\varepsilon
n/2}}\cdot C2^n$. Summing over $m$ and $j$, we obtain a bound
of the form $Ce^{-2^{\delta n}}$, which is stronger than
(\ref{eq_firststep}).
\end{pf}
\begin{cor}
\label{cor_couplagepartiel}
Let $\tilde R_n=(\tilde R_{n,j})_{j<F(n)}$ be distributed like
$\tilde X_n$ and such that the $\tilde R_n$ are independent
of each other. There then exist $C>0$ and a coupling between
$(\tilde X_1,\tilde X_2,\ldots)$ and $(\tilde R_1,\tilde
R_2,\ldots)$ such that for all $(n,j)$,
%
%e5.5 ###
%
\begin{equation}
P( |\tilde X_{n,j} -\tilde R_{n,j}|\geq C 4^{-n}) \leq C 4^{-n}.
\end{equation}
\end{cor}
\begin{pf}
By Lemma \ref{weaklimit}, it is enough to build such a coupling
between $(\tilde X_1,\ldots,\tilde X_N)$ and $(\tilde
R_1,\ldots,\tilde R_N)$ for fixed $N$ (we just have to ensure
that the constant $C$ we obtain is independent of $N$, of
course).

We use Lemma \ref{lem1kqljsdf} to get a good coupling that
makes $\tilde X_N$ independent of the other variables, then again use
this lemma to make $\tilde X_{N-1}$ independent of the
other ones and so on. In the end, this yields the desired
coupling between $\tilde X$ and $\tilde R$.

Let us be more formal. For $n\leq N$, we denote by $(\tilde
R_1^n,\ldots, \tilde R_n^n)$ a process distributed like $(\tilde
X_1,\ldots,\tilde X_n)$. Also, let $\tilde R_n$ be distributed
like $\tilde X_n$, independent of everything else.
Lemma \ref{lem1kqljsdf} and the Strassen--Dudley Theorem
\ref{thm_exists_coupling} give a good coupling between $(\tilde
R_1^n,\ldots, \tilde R_n^n)$ and $(\tilde R_1^{n-1},\ldots,
\tilde R_{n-1}^{n-1}, \tilde R_n)$. Putting all those couplings
together on a single space (by Lemma
\ref{lem_jointwocouplings}), we obtain a space on which live, in
particular, $(\tilde R_1^N,\ldots, \tilde R_N^N)$ and $(\tilde
R_1,\ldots,\tilde R_N)$, which are the processes we are trying
to couple. Moreover,
\[
|\tilde R_n^N - \tilde R_n|\leq\sum_{j=n+1}^{N}|\tilde R_n^{j} -
\tilde R_n^{j-1}|
+|\tilde R_n^n-\tilde R_n|.
\]
If $|\tilde R_n^{j} - \tilde R_n^{j-1}| \leq C4^{-j}$ for $j\in
[n+1,N]$ and $|\tilde R_n^n-\tilde R_n|\leq C4^{-n}$, then we get
$|\tilde R_n^N - \tilde R_n| \leq C' 4^{-n}$ for some constant
$C'$ independent of $n$ and $N$. In particular, $P( |\tilde
R_n^N - \tilde R_n| > C' 4^{-n})$ is bounded by
\begin{eqnarray*}
&&\sum_{j=n+1}^{N}P(|\tilde R_n^{j} - \tilde R_n^{j-1}| > C4^{-j})
+ P(|\tilde R_n^n-\tilde R_n|> C4^{-n})
\\
&&\qquad
\leq\sum_{j=n}^N C 4^{-j}
\leq C' 4^{-n}.
\end{eqnarray*}
\upqed\end{pf}
\begin{lem}
\label{couple_local}
For any $n\in\mathbb{N}$, we have
%
%e5.6 ###
%
\begin{equation}
\label{eq_secondstep}
\pi\bigl( (\tilde R_{n,j})_{0\leq j < F(n)}, (\tilde Y_{n,j})_{0\leq j
<F(n)}\bigr) < C 4^{-n},
\end{equation}
where $\tilde Y_{n,j}=Y_{n,j}+V_{n,j}$.
\end{lem}
\begin{pf}
Let $f=f(n)=\lfloor\beta n\rfloor$ and $F=2^f$. We will first
make the variables $(\tilde R_{n,j})_{j<F/2}$ independent of
the variables $(\tilde R_{n,j})_{F/2\leq j < F}$ by using the
large gap $J_{n, F/2}$, then proceed in each remaining half
using the gap in the middle of this half and so on.

We define $\tilde Y_{n,j}^i$ for $0 \leq i \leq f$ as
follows: for $0\leq k < 2^{f-i}$, the random variable $\tilde
\mathcal{Y}_{n,k}^i:=(\tilde Y_{n, j}^i)_{k2^i \leq j
<(k+1)2^i}$ is distributed like $(\tilde X_{n, j})_{k2^i \leq j
<(k+1)2^i}$, and $\tilde\mathcal{Y}_{n,k}^i$ is independent of
$\tilde\mathcal{Y}_{n,k'}^i$ if $k\not=k'$. Hence, the process
$(\tilde Y_{n,j}^f)_{0\leq j< F}$ coincides with $(\tilde
R_{n,j})_{0\leq j< F}$, while $(\tilde Y_{n,j}^0)_{0\leq j< F}$
coincides with $(\tilde Y_{n,j})_{0\leq j< F}$.

Writing $\tilde Y^i=(\tilde Y^i_{n,j})_{j<F}$, let us estimate
$\pi( \tilde Y^i, \tilde Y^{i-1})$ for $1\leq i \leq f$.
Since the variables $\tilde\mathcal{Y}^i_{n,k}$ are already
independent of one another for $0\leq k < 2^{f-i}$, we have
%
%e5.7 ###
%
\begin{equation}
\label{jlmqmlskdjfqs}
\pi( \tilde Y^i, \tilde Y^{i-1}) \leq\sum_{k=0}^{2^{f-i}-1}
\pi( \tilde\mathcal{Y}^i_{n,k}, (\tilde\mathcal
{Y}^{i-1}_{n,2k},\tilde\mathcal{Y}
^{i-1}_{n,2k+1})).
\end{equation}
Moreover, $\tilde\mathcal{Y}^i_{n,k}$ is made of $2^i$ sums of
variables $A_\ell$ along blocks, each of these blocks has
length at most $2^{n-f}$ and there is a gap $J_{n,
k2^i+2^{i-1}}$ of size $C^{\pm1} 2^{\varepsilon n+i}$ in the
middle. Therefore,\vspace*{1pt} (\ref{Hindep}) ensures that the difference
between the characteristic functions of $\tilde\mathcal{Y}^i_{n,k}$
and $(\tilde\mathcal{Y}^{i-1}_{n,2k},\tilde\mathcal
{Y}^{i-1}_{n,2k+1})$ is
at most
\[
C (1+2^{n-f})^{C 2^i} e^{-c 2^{\varepsilon n+i}}
\leq C e^{Cn 2^i -c 2^{\varepsilon n+i}}
\leq C e^{-c' 2^{\varepsilon n+i}},
\]
if $n$ is large enough. Taking $N=2^i$ and $T'=e^{2^{\varepsilon
n/2}}$ in Lemma \ref{lemsmoothbasic}, we obtain (with
computations very similar to those in the proof of Lemma
\ref{lem1kqljsdf})
\[
\pi( \tilde\mathcal{Y}^i_{n,k}, (\tilde\mathcal
{Y}^{i-1}_{n,2k},\tilde\mathcal{Y}
^{i-1}_{n,2k+1}))
\leq C e^{- 2^{\delta n}}
\]
for some $\delta>0$. Summing over $k$ in (\ref{jlmqmlskdjfqs})
and then over $i$, we obtain
\[
\pi(\tilde Y^f, \tilde Y^0) \leq\sum_{i=1}^f \pi( \tilde Y^i,
\tilde Y^{i-1})
\leq f 2^f C e^{- 2^{\delta n}} \leq C e^{- 2^{\delta n}/2}.
\]
This gives, in particular, (\ref{eq_secondstep}).
\end{pf}
\begin{pf*}{Proof of Proposition \ref{prop_make_all_indep}}
We combine the coupling constructed in Corollary
\ref{cor_couplagepartiel} with the couplings constructed in
Lemma \ref{couple_local} for each $n$, using the
Strassen--Dudley Theorem \ref{thm_exists_coupling}. We obtain a
coupling between $(\tilde X_{n,j})$ and $(\tilde Y_{n,j})$ such
that $P(|\tilde X_{n,j}-\tilde Y_{n,j}|\geq C4^{-n}) \leq
C4^{-n}$. Since $\sum_{n,j} 4^{-n}<\infty$, the Borel--Cantelli
lemma ensures that almost surely
%
%e5.8 ###
%
\begin{equation}
\label{eq_avec_tildes}
\sup_{(n,j)} \biggl|\sum_{(n',j')\prec(n,j)} \tilde X_{n',j'} -
\tilde Y_{n',j'} \biggr| <\infty.
\end{equation}
Moreover, $\tilde X_{n,j}=X_{n,j}+V_{n,j}$, where the random
variables $V_{n,j}$ are centered, independent and in $L^2$. By
the law of the iterated logarithm, almost surely, for any
$\alpha>0$,
\[
\biggl| \sum_{(n',j')\prec(n,j)} V_{n',j'} \biggr|=o\bigl( \operatorname
{Card}\{
(n',j') \mid(n',j')\prec(n,j)\}^{1/2+\alpha}\bigr).
\]
Moreover, $\operatorname{Card}\{(n',j') \mid(n',j')\prec(n,j)\} \leq
\sum_{n'=1}^n \sum_{j'<F(n')} 1 \leq C 2^{\beta n}$. We obtain
almost surely
\[
\biggl| \sum_{(n',j')\prec(n,j)} X_{n',j'}-\tilde X_{n',j'} \biggr|
=o\bigl( 2^{\beta n (1/2+\alpha)}\bigr).
\]
A similar estimate holds for $Y_{n,j}-\tilde Y_{n,j}$. With
(\ref{eq_avec_tildes}), this proves the proposition.
%\rightqed
\end{pf*}

%s5.2 ###
\subsection{Step 2: Coupling with Gaussian random vectors}
We are going to use Corollary 3 of Za{\u\i}tsev \cite{zaitsevLp}.
Let us
recall it here, for the convenience of the reader, in a form that
is suitable for us (it is obtained from the statement of
Za{\u\i}tsev by taking $r=10/e$, $\gamma=q$, $L_\gamma=M^q$,
$n=b$ and $z'=Mz/5$).
\begin{prop}
\label{prop_zaitsev}
Let $Y_0,\ldots,Y_{b-1}$ be independent centered $\mathbb{R}^d$-valued
random vectors. Let $q\geq2$ and
$M= (\sum_{j=0}^{b-1} E |Y_j|^q )^{1/q}$. Assume\vspace*{1pt} that
there exists a sequence $0=m_0<m_1<\cdots<m_s=b$ satisfying the
following condition. Letting $\zeta_k=Y_{m_{k}}+\cdots+
Y_{m_{k+1}-1}$ and $B_k=\operatorname{cov}\zeta_k$, we assume that
%
%e5.9 ###
%
\begin{equation}
\label{qsldkfjlmqksfd}
100 M^2|v|^2 \leq\langle B_k v, v\rangle\leq100C M^2 |v|^2
\end{equation}
for all $v\in\mathbb{R}^d$, all $0\leq k< s$ and some constant
$C\geq1$. There then exists a coupling between
$(Y_0,\ldots,Y_{b-1})$ and a sequence of independent Gaussian
random vectors $(S_0,\ldots, S_{b-1})$ such that $\operatorname
{cov}S_j=\operatorname{cov}
Y_j$ and, moreover,
%
%e5.10 ###
%
\begin{equation}
\label{qsidfqsdfqsfdqsdf}
P \Biggl( \max_{1\leq i\leq b} \Biggl|\sum_{j=0}^{i-1} Y_j -S_j
\Biggr| \geq M z \Biggr)
\leq C' z^{-q} + \exp(-C'z )
\end{equation}
for all $z\geq C' \log n$. Here, $C'$ is a positive quantity
depending only on $C$, the dimension $d$ and the integrability
exponent $q$.
\end{prop}

The following lemma easily follows from the previous
proposition.
\begin{lem}
\label{lem_cor_zaitsev}
For $n\in\mathbb{N}$, there exists a coupling between
$(Y_{n,0},\ldots,Y_{n,F(n)-1})$ and $(S_{n,0},\ldots,
S_{n,F(n)-1})$, where the $S_{n,j}$ are independent centered
Gaussian vectors with $\operatorname{cov}S_{n,j}=\operatorname
{cov}Y_{n,j}$, such that
%
%e5.11 ###
%
\begin{equation}
\sum_n P \Biggl( \max_{1\leq i\leq F(n)} \Biggl|\sum_{j=0}^{i-1}
Y_{n,j} -S_{n,j} \Biggr|
\geq2^{((1-\beta)/2+\beta/p+\varepsilon/2)n} \Biggr)
<\infty.
\end{equation}
\end{lem}
\begin{pf}
Let $q\in(2,p)$ and $n\in\mathbb{N}$. We want to apply Proposition
\ref{prop_zaitsev} to the independent vectors $(Y_{n,j})_{0\leq
j< F}$, where $F=F(n)=2^{\lfloor\beta n\rfloor}$.

By Proposition \ref{Lpbound}, we have $\| Y_{n,j} \|_{L^q}\leq
C 2^{(1-\beta)n/2}$. This implies that $M:=
(\sum_{j=0}^{F-1} E |Y_{n,j}|^q )^{1/q}$ satisfies
%
%e5.12 ###
%
\begin{equation}
\label{borneM}
M \leq C 2^{\beta n/q}\cdot2^{(1-\beta)n/2}.
\end{equation}

By the assumptions of Theorem \ref{main_thm}, $\operatorname{cov}
Y_{n,j}=|I_{n,j}|\Sigma^2+o( |I_{n,j}|^\alpha)$ for any
$\alpha>0$. In particular,
%
%e5.13 ###
%
\begin{equation}
\operatorname{cov}Y_{n,j}=2^{(1-\beta)n} \Sigma^2 \bigl(1+o(1)\bigr).
\end{equation}
Moreover, we assume that the matrix $\Sigma^2$ is nondegenerate.
Therefore, there exists a constant $C_0$ such that
%for any
%$0\leq m<m'\leq F$ with $m'-m$ large enough and for
for any large enough $n$, any $0\leq m < m' < F(n)$ and
any vector
$v$,
we have
\[
C_0^{-1}(m'-m) 2^{(1-\beta)n} |v|^2 \leq\Biggl\langle\sum
_{j=m}^{m'-1} \operatorname{cov}Y_{n,j} v,v \Biggr\rangle
\leq C_0(m'-m) 2^{(1-\beta)n} |v|^2.
\]
For $m=0$ and $m'=F$, the quantity
$(m'-m)2^{(1-\beta)n}=2^{\lfloor\beta n\rfloor} \cdot
2^{(1-\beta)n}$ is much larger than $M^2$, by (\ref{borneM}).
On the other hand, each individual term (for $m'=m+1$) is
bounded by
\[
|{\operatorname{cov}}Y_{n,j}| |v|^2 \leq\| Y_{n,j} \|_{L^2}^2 |v|^2
\leq
\| Y_{n,j} \|_{L^q}^2 |v|^2 \leq M^2 |v|^2.
\]
Therefore, we can group the $Y_{n,j}$ into consecutive blocks
so that (\ref{qsldkfjlmqksfd}) is satisfied for some constant
$C$.

Applying Proposition \ref{prop_zaitsev}, we get a coupling
between $(Y_{n,0},\ldots,Y_{n,F-1})$ and $(S_{n,0},\ldots,
S_{n,F-1})$ such that
%
%e5.14 ###
%
\begin{equation}
P \Biggl( \max_{1\leq i\leq F} \Biggl|\sum_{j=0}^{i-1} Y_{n,j}
-S_{n,j} \Biggr| \geq2^{\varepsilon n/3} M \Biggr)
\leq C 2^{-q\varepsilon n/3}
\end{equation}
by (\ref{qsidfqsdfqsfdqsdf}), for $z=2^{\varepsilon n/3}$. This
quantity is summable in $n$. Since $2^{\varepsilon n/3} M \leq
2^{((1-\beta)/2+\beta/p+\varepsilon/2)n}$ if $q$ is close enough
to $p$ and $n$ is large enough, this completes the proof of the
lemma.
\end{pf}
\begin{lem}
\label{lem_make_var_equal}
Let $Z_{n,j}$ be independent Gaussian random vectors such that
$\operatorname{cov}Z_{n,j}=|I_{n,j}|\Sigma^2$. There then exists a coupling
between $(S_{n,j})$ and $(Z_{n,j})$ such that almost surely
%
%e5.15 ###
%
\begin{equation}
\label{eq_erreurs_variance}
\sum_{(n',j')\prec(n,j)} S_{n',j'}-Z_{n',j'}=o\bigl( 2^{ (\beta+\varepsilon)n/2}\bigr).
\end{equation}
\end{lem}
\begin{pf}
Let $\alpha>0$. Let $E_{n,j}=\mathcal{N}(0, |I_{n,j}|\Sigma^2 +
2^{\alpha n}I_d)$, where $I_d$ is the identity matrix of
dimension $d$. By assumption, $\operatorname{cov}
S_{n,j}=|I_{n,j}|\Sigma^2+o( 2^{\alpha n})$. In particular, if
$n$ is large enough, we can write $|I_{n,j}|\Sigma^2+ 2^{\alpha
n}I_d=\operatorname{cov}S_{n,j}+ M_{n,j}$, where the matrix $M_{n,j}$ is
positive definite and $|M_{n,j}|\leq C 2^{\alpha n}$.
Therefore, $E_{n,j}$ is the sum of $S_{n,j}$ and an independent
random variable distributed like $\mathcal{N}(0, M_{n,j})$. In the
same way, $E_{n,j}$ is the sum of $Z_{n,j}$ and of an
independent Gaussian $\mathcal{N}(0, 2^{\alpha n} I_d)$. Putting those
decompositions on a single space, using Lemma
\ref{lem_jointwocouplings}, we obtain a coupling between
$(S_{n,j})$ and $(Z_{n,j})$ such that the difference
$D_{n,j}=S_{n,j}-Z_{n,j}$ is centered and where
$\| D_{n,j} \|_{L^2} \leq C 2^{\alpha n/2}$.

We claim that this coupling satisfies the conclusion of the
lemma if $\alpha<\varepsilon/2$. Indeed, by Etemadi's inequality
(\cite{billinsgleyconvergence}, Paragraph M19), we have, for any
$n$,
\begin{eqnarray*}
&&
P \Biggl( {\max_{1\leq i \leq F(n)}} \Biggl|\sum_{j=0}^{i-1}
D_{n,j} \Biggr|
> 2^{(\beta+\varepsilon/2)n/2} \Biggr)
\\
&&\qquad
\leq C \max_{1\leq i \leq F(n)} P \Biggl( \Biggl|\sum_{j=0}^{i-1}
D_{n,j} \Biggr|
> 2^{(\beta+\varepsilon/2)n/2}/3 \Biggr)
\\
&&\qquad
\leq C \max_{1\leq i \leq F(n)} E \Biggl( \Biggl| \sum_{j=0}^{i-1}
D_{n,j} \Biggr|^2 \Biggr)
\bigg/2^{(\beta+\varepsilon/2)n}
\\
&&\qquad
\leq C \sum_{j=0}^{F(n)-1} E(|D_{n,j}|^2) / 2^{(\beta+\varepsilon/2)n}
\leq C 2^{\beta n} 2^{\alpha n} / 2^{(\beta+\varepsilon/2)n}.
\end{eqnarray*}
This is summable. Therefore, almost surely for large enough
$n$ and for $1\leq i\leq F(n)$, we have $ |\sum_{j=0}^{i-1}
D_{n,j} |\leq2^{(\beta+\varepsilon/2)n/2}$. The estimate
(\ref{eq_erreurs_variance}) follows.
\end{pf}

Putting together the couplings constructed in Lemmas
\ref{lem_cor_zaitsev} and \ref{lem_make_var_equal}, we obtain a
coupling satisfying the conclusions of Step 2.

%s5.3 ###
\subsection{Step 4: Handling the maxima}

We recall that $i_{n,j}$ is the smallest element of the
interval $I_{n,j}$.
\begin{lem}
Almost surely when $(n,j)\to\infty$,
\[
\max_{m<|I_{n,j}|} \Biggl|\sum_{\ell=i_{n,j}}^{i_{n,j}+m} A_\ell
\Biggr| = o\bigl(2^{((1-\beta)/2+\beta/p+\varepsilon)n}\bigr).
\]
\end{lem}
\begin{pf}
Let $q\in(2,p)$. In $L^{q}$, the partial sums
$\sum_{\ell=a}^{b-1} A_\ell$ are bounded by $C(b-a)^{1/2}$, by
Proposition \ref{Lpbound}. Let $M_a^b={\max_{a\leq n \leq b}}
| {\sum_{\ell=a}^{n-1} A_\ell}|$. Corollary B1 in
Serfling \cite{serfling} then also shows that
%
%e5.16 ###
%
\begin{equation}
\label{controlemax}
\| M_a^b \|_{L^{q}} \leq C (b-a)^{1/2}
\end{equation}
for a different constant $C$. In particular, if
$\nu=(1-\beta)/2+\beta/p+\varepsilon/2$,
then
\begin{eqnarray*}
P\bigl( M_{i_{n,j}}^{i_{n,j}+|I_{n,j}|} \geq2^{\nu n}\bigr)
&\leq&
E \bigl( \bigl(M_{i_{n,j}}^{i_{n,j}+|I_{n,j}|}\bigr)^{q} \bigr)/2^{ \nu nq}
\\
&\leq& C |I_{n,j}|^{q/2}/2^{ \nu nq}.
\end{eqnarray*}
Moreover,
\[
\sum_{n,j} |I_{n,j}|^{q/2}/2^{ \nu nq}
\leq\sum_n 2^{\beta n}\cdot2^{ (1-\beta)n q/2 - \nu nq}.
\]
This sum is finite if $q$ is close enough to $p$.
The Borel--Cantelli lemma gives the desired result.
\end{pf}

%s5.4 ###
\subsection{Step 6: The gaps}

Recall that $\mathcal{J}$ is the union of the gaps $J_{n,j}$. In this
subsection, we prove the following lemma.
\begin{lem}
\label{step6_mainlemma}
For any $\alpha>0$, there exists $C>0$ such that for any
interval $J\subset\mathbb{N}$,
\[
E \biggl|\sum_{\ell\in J\cap\mathcal{J}} A_\ell\biggr|^2 \leq C
|J\cap
\mathcal{J}|^{1+\alpha}.
\]
\end{lem}

Together with the Gal--Koksma strong law of large numbers
(Proposition \ref{prop_galkoksma}) applied with $q=1+\alpha$,
this shows that for every $\alpha>0$, almost surely
\[
\sum_{\ell< k, \ell\in\mathcal{J}} A_\ell=o( |\mathcal{J}\cap
[0,k]|^{1/2+\alpha}).
\]
Moreover, for $k\in[2^n, 2^{n+1})$, we have [by
(\ref{sommeinterv})]
\[
|\mathcal{J}\cap[0,k]| \leq\sum_{m=1}^n \sum_{j=0}^{F(m)-1} |J_{m,j}|
\leq C\sum_{m=1}^n m 2^{\varepsilon m +\beta m}
\leq C n 2^{\varepsilon n+\beta n}
\leq C k^{ \beta+3\varepsilon/2}.
\]
With the previous equation, we obtain (if $\alpha$ is small
enough)
\[
\sum_{\ell< k, \ell\in\mathcal{J}} A_\ell
=
o(k^{\beta/2+\varepsilon}).
\]
This is (\ref{step6}), as desired.
\begin{pf*}{Proof of Lemma \ref{step6_mainlemma}}
We will freely use the convexity inequality
%
%e5.17 ###
%
\begin{equation}
\label{convtriv}
(a_1+\cdots+a_k)^2 \leq k(a_1^2+\cdots+a_k^2).
\end{equation}
Let $J\subset\mathbb{N}$ be an interval. We decompose $J\cap\mathcal
{J}$ as
$J_0\cup J_1 \cup J_2$, where $J_0$ and $J_2$ are,
respectively, the first and the last interval of $J\cap\mathcal{J}$,
and $J_1$ is the remaining part (it is therefore a union of
full intervals of $\mathcal{J}$). Then
\[
\biggl|\sum_{\ell\in J\cap\mathcal{J}} A_\ell\biggr|^2
\leq3 \biggl|\sum_{\ell\in J_0} A_\ell\biggr|^2
+ 3 \biggl|\sum_{\ell\in J_1} A_\ell\biggr|^2
+ 3 \biggl|\sum_{\ell\in J_2} A_\ell\biggr|^2.
\]
The set $J_0$ is an interval, hence Proposition \ref{Lpbound}
gives $E |\sum_{\ell\in J_0} A_\ell|^2 \leq C |J_0|$.
A~similar inequality holds for $J_2$. To conclude the proof, it
is therefore sufficient to prove that
%
%e5.18 ###
%
\begin{equation}
\label{mlkwjvxcopiuwvx}
E \biggl|\sum_{\ell\in J_1} A_\ell\biggr|^2
\leq C |J_1|^{1+\alpha}.
\end{equation}
Since $J_1$ is not always an interval, this does not follow
directly from Proposition~\ref{Lpbound}. However, this is
trivial if $J_1$ is empty. Otherwise, let $N$ be such that
$\max J_1 \in[2^N, 2^{N+1})$. Since the last interval in $J_1$
is contained in $[2^N,2^{N+1})$, its length is $2^{\lfloor
\varepsilon N \rfloor+r}$ for some $r\in[0, f(N)]$. In
particular, $N\leq C {\log}|J_1|$.

We defined the notion of rank of an interval $J_{n,j}$ in the paragraph
before equation (\ref{sommeinterv}): such an interval
has rank $r\in[0, f(n)]$ if its length is $2^{\lfloor\varepsilon
n\rfloor+ r}$. There are $2^{f(n)-1-r}$ intervals of rank $r$
in $[2^n, 2^{n+1})$ for $r<f(n)$ and one interval of rank
$f(n)$.

For $n\in\mathbb{N}$ and $0\leq r\leq f(n)$, let $\mathcal
{J}^{(n,r)}$ denote
the union of the intervals $J_{n,j}$ which are of rank $r$. The
number of sets $\mathcal{J}^{(n,r)}$ intersecting $J_1$ is at most
$\sum_{n=0}^N (f(n)+1) \leq CN^2$. Hence, by the convexity
inequality (\ref{convtriv}),
%
%e5.19 ###
%
\begin{equation}
\label{qmlskdfjpoi}
\biggl|\sum_{\ell\in J_1} A_\ell\biggr|^2
\leq CN^2 \sum_{n,r} \biggl|\sum_{\ell\in J_1 \cap\mathcal{J}
^{(n,r)}}A_\ell\biggr|^2.
\end{equation}

Let us fix some $(n,r)$ and enumerate the intervals of
$\mathcal{J}^{(n,r)}$ as $K_1,\ldots, K_t$ for $t=2^{f(n)-1-r}$ if
$r<f(n)$ [or $t=1$ if $r=f(n)$]. Let $T_s=\sum_{\ell\in K_s}
A_\ell$. We claim that for any subset $S$ of $\{1,\ldots,t\}$,
%
%e5.20 ###
%
\begin{equation}
\label{indepgaps}
E \biggl| \sum_{s\in S} T_s \biggr|^2 \leq C \sum_{s\in S} E|T_s|^2+C |S|.
\end{equation}
Let us show how this completes the proof. By Proposition
\ref{Lpbound}, we have $E|T_s|^2 \leq C |K_s|$. Therefore, for
any set $K$ which is a union of intervals in $\mathcal{J}^{(n,r)}$, we
obtain $E|\sum_{\ell\in K} A_\ell|^2 \leq C |K|$. This applies,
in particular, to $K=J_1 \cap\mathcal{J}^{(n,r)}$. Therefore,
(\ref{qmlskdfjpoi}) gives
\[
E \biggl|\sum_{\ell\in J_1} A_\ell\biggr|^2
\leq CN^2 \sum_{n,r} \bigl|J_1 \cap\mathcal{J}^{(n,r)}\bigr|
=CN^2 |J_1|.
\]
Together with the inequality $N\leq C \log|J_1|$, this proves
(\ref{mlkwjvxcopiuwvx}), as desired.

It remains to prove (\ref{indepgaps}). We first make the $T_s$
independent, as follows. Let $(U_1,\ldots,U_t)$ be independent
random variables such that $U_s$ is distributed like $T_s$.
Also, let $V_1,\ldots,V_t,V'_1,\ldots, V'_t$ be independent random
variables distributed like $V$ (constructed in Proposition
\ref{prop_V}) and write $\tilde T_s=T_s+V_s$, $\tilde
U_s=U_s+V'_s$. We claim that for some $\delta>0$ and $C>0$,
%
%e5.21 ###
%
\begin{equation}
\pi( (\tilde T_1,\ldots, \tilde T_t), (\tilde U_1,\ldots,\tilde U_t))
< C e^{-2^{\delta n}}.
\end{equation}
To prove this estimate, we use the intervals of rank greater than $r$ as
gaps: we first make $\tilde T_1,\ldots,\tilde T_{t/2}$
independent of $\tilde T_{t/2+1},\ldots, \tilde T_{t}$ using the
gap $J_{n, F/2}$, then proceed in each half using the central
gaps $J_{n, F/4}$ and $J_{n, 3F/4}$, and so on. The details of
the argument are exactly the same as in the proof of
Lemma \ref{couple_local}.

Thanks to this estimate and the Strassen--Dudley Theorem
\ref{thm_exists_coupling}, we can construct a coupling between
$(\tilde T_j)_{1\leq j\leq t}$ and $(\tilde U_j)_{1\leq j\leq
t}$ such that, outside a set $O$ of measure at most $C
e^{-2^{\delta n}}$, we have $|\tilde T_j - \tilde U_j|\leq C
e^{-2^{\delta n}}$ for $1\leq j\leq t$. For any subset $S$ of
$\{1,\ldots, t\}$, we obtain (as in the proof of Lemma
\ref{lemborneL2})
\[
\biggl\| \sum_{s\in S} \tilde T_s \biggr\|_{L^2}
\leq\biggl\| 1_O \sum_{s\in S} \tilde T_s \biggr\|_{L^2}
+ \biggl\| 1_{O^c} \sum_{s\in S} \tilde T_s -\tilde U_s \biggr\|_{L^2}
+ \biggl\| \sum_{s\in S} \tilde U_s \biggr\|_{L^2}.
\]
The first term is bounded by $\| 1_O \|_{L^q}\| \sum_{s\in S} \tilde
T_s \|_{L^p}$, where $q$ is chosen such that
$1/p+1/q=1/2$. Hence, it is at most $C e^{-2^{\delta n}/q} 2^n
\leq C$. The second term is bounded by $C te^{-2^{\delta n}}
\leq C$. Therefore, $\| \sum_{s\in S} T_s \|_{L^2}$ is bounded
by
\[
\biggl\| \sum_{s\in S} \tilde T_s \biggr\|_{L^2}
+\biggl\| \sum_{s\in S}V_s \biggr\|_{L^2}
\leq C + \biggl\| \sum_{s\in S} U_s \biggr\|_{L^2}
+ \biggl\| \sum_{s\in S}V_s \biggr\|_{L^2}
+ \biggl\| \sum_{s\in S}V'_s \biggr\|_{L^2}.
\]
Since the $U_s$ are centered independent random variables,
$\| \sum_{s\in S} U_s \|_{L^2} =\break (\sum
E(U_s^2) )^{1/2} = ( \sum E(T_s^2) )^{1/2}$. In
the same way, we have $\| \sum_{s\in S}V_s \|_{L^2} =
\| \sum_{s\in S}V'_s \|_{L^2} = C |S|^{1/2}$. We get
$\| \sum_{s\in S} T_s \|_{L^2} \leq C + ( \sum
E(T_s^2) )^{1/2} + C|S|^{1/2}$, which implies
(\ref{indepgaps}).
\end{pf*}

%s6 ###
\section{Completing the proof of the main theorems}

In this section, we first finish the proof of Theorem
\ref{main_thm} when the matrix $\Sigma^2$ is degenerate and
then derive Theorem \ref{main_thm_stat} from Theorem
\ref{main_thm}.
\begin{lem}
\label{lemzero}
Let $(A_0,A_1,\ldots)$ be a process satisfying the assumptions
of Theorem \ref{main_thm} for $\Sigma^2=0$. Then almost
surely $\sum_{\ell=0}^{n-1} A_\ell=o(n^\lambda)$ for any
$\lambda>p/(4p-4)$.
\end{lem}
\begin{pf}
Let $\beta>0$ and $\varepsilon>0$. Define a sequence of intervals
$I_n=[n^{\beta+1}, (n+1)^{\beta+1})\cap\mathbb{N}$ and denote by
$i_n=\lceil n^{\beta+1} \rceil$ the smallest element of $I_n$.
We claim that almost surely
%
%e6.1 ###
%
\begin{equation}
\label{qslkdfjsxv}
\Biggl|\sum_{\ell=0}^{i_n-1} A_\ell\Biggr|=O(n^{1/2+\varepsilon})
\end{equation}
and
%
%e6.2 ###
%
\begin{equation}
\label{qslkdfjsxv2}
\max_{i\in I_n} \Biggl|\sum_{\ell=i_n}^i A_\ell\Biggr| =
O(n^{\beta/2+1/p+\varepsilon}).
\end{equation}
Taking $\beta=(p-2)/p$ to equate the error terms, we get
$ |{\sum_{\ell\leq k} A_\ell}|=O(n^{1/2+\varepsilon})$,
where $n=n(k)$ is the index of the interval $I_n$ containing
$k$. Since $n \leq Ck^{1/(1+\beta)}$, we finally obtain an
error term $O(k^{\lambda+2\lambda\varepsilon})$ for
\[
\lambda=\frac{1}{2}\cdot\frac{1}{1+(p-2)/p}=\frac{p}{4p-4}.
\]
This concludes the proof. It remains to establish
(\ref{qslkdfjsxv}) and (\ref{qslkdfjsxv2}).

By (\ref{covOK}), $\|{ \sum_{\ell=0}^{i_n-1} A_\ell}\|
_{L^2}=O(n^{\alpha/2})$ for any $\alpha>0$. Therefore,
\[
P \Biggl(\sum_{\ell=0}^{i_n-1} A_\ell\geq n^{1/2+\varepsilon} \Biggr)
\leq
\Biggl\| \sum_{\ell=0}^{i_n-1} A_\ell\Biggr\|_{L^2}^2\bigg/n^{1+2\varepsilon}
\leq C n^\alpha/n^{1+2\varepsilon}.
\]
Taking $\alpha=\varepsilon$, this quantity is summable.
Equation (\ref{qslkdfjsxv}) follows.

Let $M_a^b={\max_{a\leq n \leq b}} | {\sum_{\ell=a}^{n-1}
A_\ell}|$. For $q<p$, we have
\begin{eqnarray*}
P \Biggl(\max_{i\in I_n} \Biggl|\sum_{\ell=i_n}^i A_\ell\Biggr|
\geq n^{\beta/2+1/p+\varepsilon} \Biggr)
&=&P(M_{i_n}^{i_{n+1}} \geq n^{\beta/2+1/p+\varepsilon})
\\
&\leq&\| M_{i_n}^{i_{n+1}} \|_{L^{q}}^{q}/n^{ q(\beta
/2+1/p+\varepsilon)}.
\end{eqnarray*}
By (\ref{controlemax}), $\| M_{i_n}^{i_{n+1}} \|_{L^{q}} \leq
C(i_{n+1}-i_n)^{1/2} \leq C n^{\beta/2}$. Therefore, the last
equation is bounded by $C/n^{q(1/p+\varepsilon)}$. This is
summable if $q$ is close enough to $p$. The estimate
(\ref{qslkdfjsxv2}) follows.
\end{pf}

Let $(A_0,A_1,\ldots)$ be a process satisfying the assumptions
of Theorem \ref{main_thm} for some matrix $\Sigma^2$. Replacing
$A_\ell$ by $A_\ell-E(A_\ell)$, we can assume that this process
is centered. We decompose $\mathbb{R}^d$ as an orthogonal sum $E\oplus
F$, where $\Sigma^2$ is nondegenerate on $E$ and vanishes on
$F$. The almost sure invariance principle along $E$ is proved
in Section \ref{secnondeg}, while Lemma \ref{lemzero} handles
$F$. This proves Theorem \ref{main_thm}.

Finally, Theorem \ref{main_thm_stat} follows directly from
Lemma \ref{covsommeborne} and Theorem \ref{main_thm}.

\printaddresses

\end{document}